\documentclass[12pt]{amsart}

\usepackage{fullpage,graphicx,psfrag,amsmath,amsfonts,verbatim, hyperref, amssymb, caption, subcaption,booktabs}

\newcommand{\R}{\mathbf{R}}

\newcommand{\Z}{\mathbf{Z}}
\newcommand{\N}{\mathbf{N}}

\newcommand{\symm}{{\mbox{\bf S}}}  
\newcommand{\diam}{{\mbox{diam}}}  

\newcommand{\conv}{\mathop{\bf conv}}

\newcommand{\E}{\mathop{\bf E{}}}

\newcommand{\Co}{{\mathop {\bf Co}}} 

\newtheorem{theorem}{Theorem}[section]
\newtheorem{prop}{Proposition}[section]

\newtheorem{lemma}[theorem]{Lemma}

\title{Convergence of the random Abelian sandpile}
\author{Ahmed Bou-Rabee}

\begin{document}

\begin{abstract}
We prove that Abelian sandpiles with random initial states converge almost surely to unique scaling limits.  The proof follows the Armstrong-Smart program for stochastic homogenization of uniformly elliptic equations.  

Using simple random walk estimates, we prove an analogous result for the divisible sandpile and identify its scaling limit as exactly that of the averaged divisible sandpile. As a corollary, this gives a new quantitative proof of known results on the stabilizability of Abelian sandpiles. 
\end{abstract}

	\maketitle

\begin{figure}[!ht]
\begin{centering}
 \includegraphics[width=0.4\textwidth]{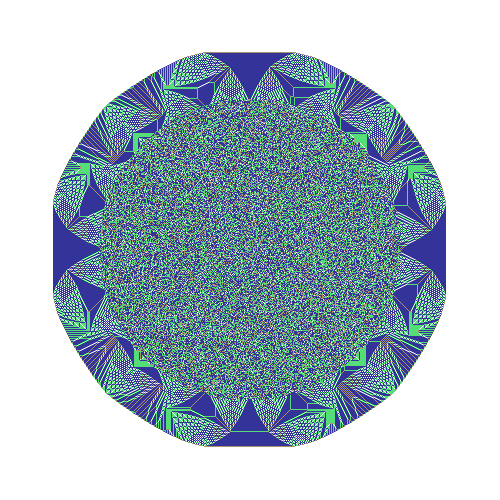}
 \includegraphics[width=0.4\textwidth]{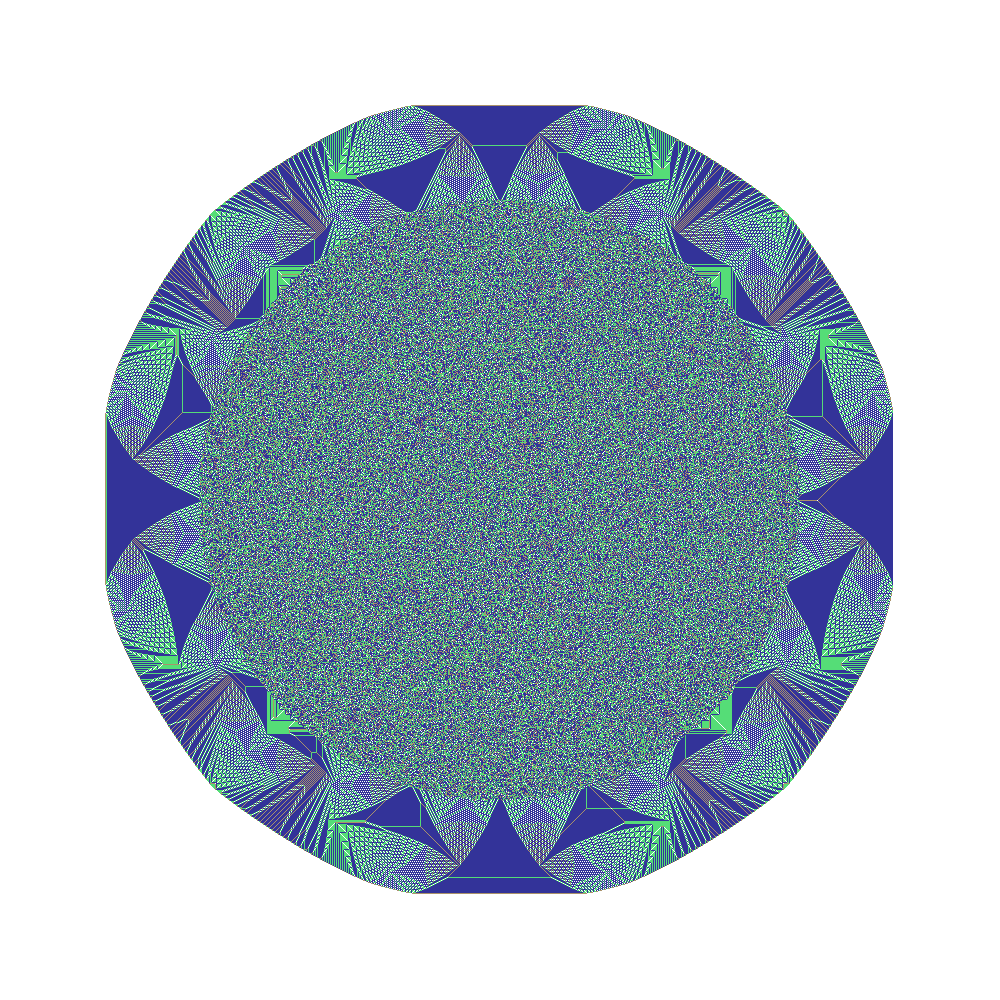} \\
 \includegraphics[width=0.4\textwidth]{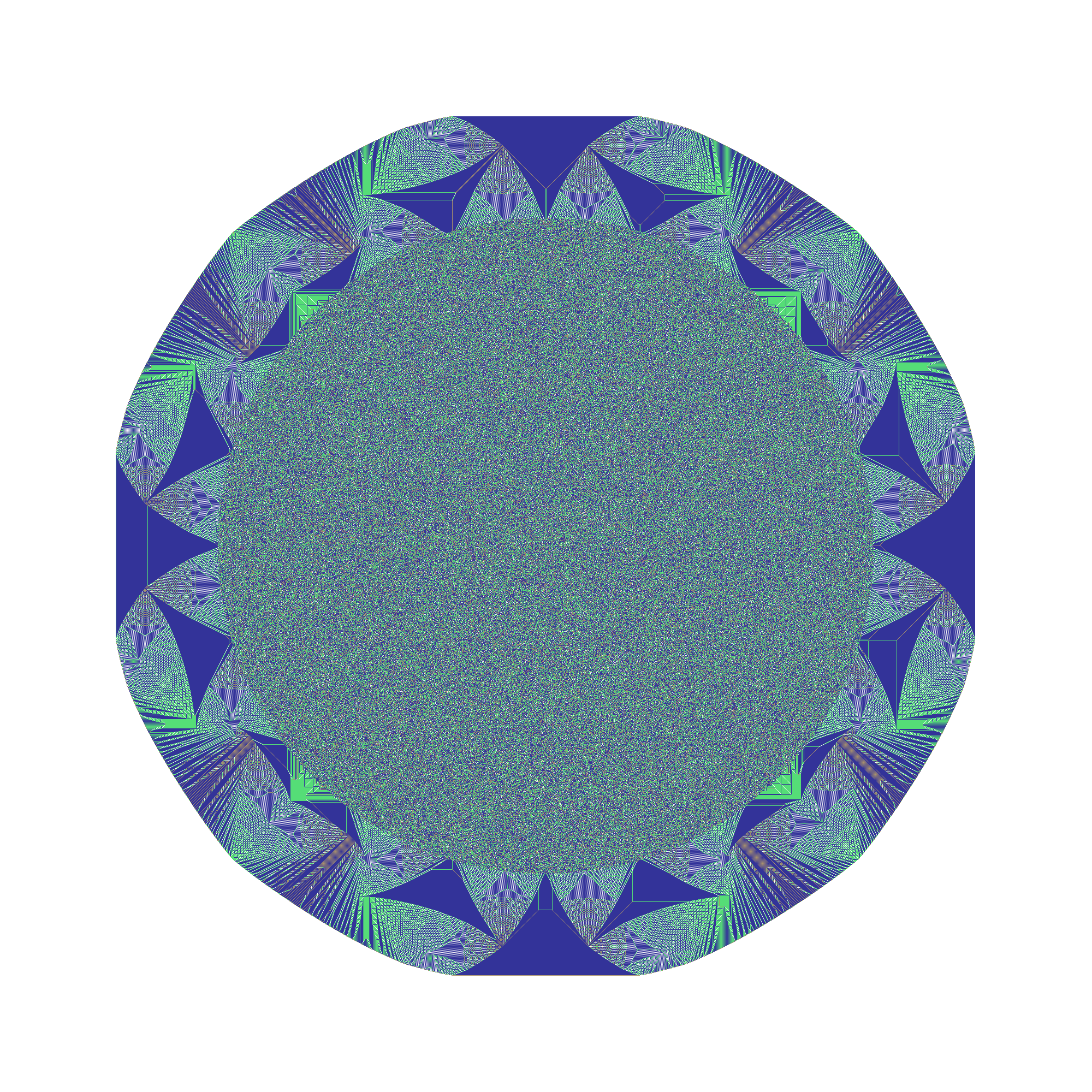}
  \includegraphics[width=0.4\textwidth]{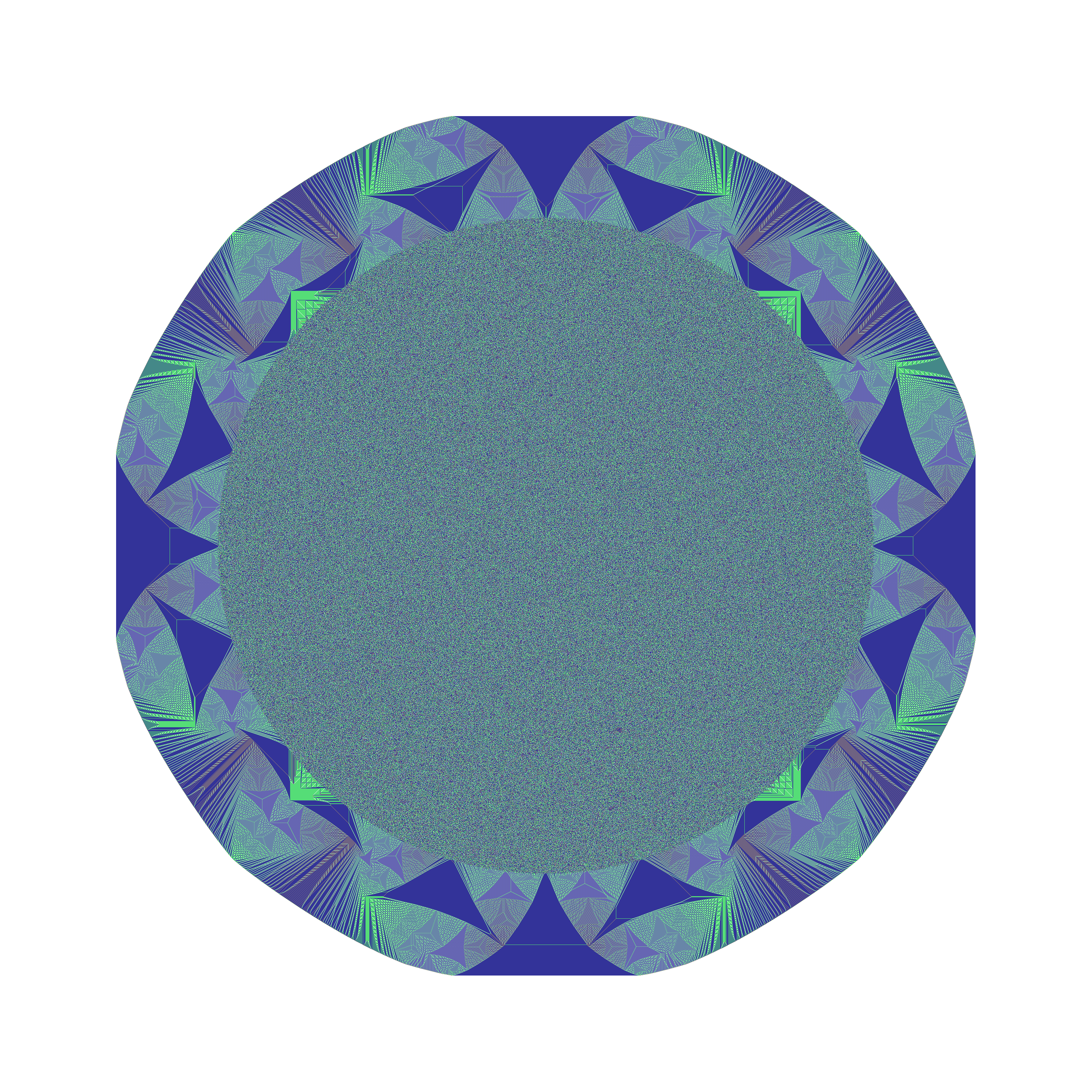}
 \caption{ For each $x$ in a ball of radius $n = 500,1000,2000,$ and 3000, flip a fair coin. If it lands heads, add 3 chips to $x$, otherwise add 5 chips. Then, stabilize. Sites with 0,1,2, and 3 chips are represented by white, brown, green, and blue respectively. }
\end{centering}
\end{figure}

\section{Introduction}
	The Abelian sandpile is a simple combinatorial model which produces striking, fractal-like patterns \cite{bak1987self,levine2010sandpile}.	
	Start with an initial {\it sandpile}, a function $\eta: \Z^d \to \Z$, which can be thought of as a configuration 
	of indistinguishable chips or grains. Whenever a site has more chips than it has neighbors, it is {\it unstable} and {\it topples}, giving one chip to each of its $2 d$ neighbors. A sandpile is {\it stabilized} by toppling unstable sites until every site has fewer chips than it has neighbors. If the initial number of chips is finite, this process terminates and the final arrangement of chips does not depend on the order in which unstable sites topple. 
	
	If you start with a large stack of chips at the origin in $\Z^d$ and stabilize, intricate, kaleidoscopic patterns appear. W. Pegden and C. Smart began the rigorous understanding of these patterns by showing that scaling limits of sandpiles exist and are Laplacians of solutions to elliptic obstacle problems \cite{pegden2013convergence}.  Their proof technique is flexible: it was first applied to the single source sandpile and it works for essentially any sandpile with a periodic initial configuration. However, their proof does not extend to random initial configurations. In this paper, as a first step towards understanding random sandpiles, we show, using a novel approach, that sandpiles with random initial states also have scaling limits.

	As a simple example, consider the following random sandpile on $\Z^d$. Start with $2d-1$ chips at
	each site in a ball of radius $n$.  Flip a fair coin for each $x$ in the ball, if the coin lands
	heads, add two extra chips at $x$. Once the initial sandpile has been set, stabilize. 

	 If you repeat this experiment for large $n$ and rescale, a non-random pattern emerges.  The pattern looks remarkably similar to the scaling limit of the single source sandpile -  compare Figures 1 and 2.  Our main result explains this similarity by proving that the scaling limit of the random sandpile is the Laplacian of the solution to an elliptic obstacle problem with two operators. One operator depends on the distribution of the randomness. The other operator is the exact same one appearing in the scaling limit of the single source sandpile. 

	More generally let $\eta: \Z^d \to \Z$ be stationary, ergodic, bounded, and satisfy $\E(\eta(0)) > 2 d-1$. Let $W \subset \R^d$ be a bounded Lipschitz domain. For each $n \in \N$, let $W_n = \Z^d \cap n W$ denote the finite difference approximation of $W$.  Initialize the sandpile according to $\eta$ in $W_n$ and set it to be 0 elsewhere. Then, stabilize, counting how many times each site topples with the {\it odometer function} $v_n : \Z^d \to \N$. Denote the stable sandpile by $s_n: \Z^d \to \Z$. Our main result is the following.

\begin{theorem} \label{thetheorem}
Almost surely, as $n \to \infty$, the rescaled functions $\bar{v}_n(x):= n^{-2} v_n([n x])$ converge uniformly to the unique solution of the elliptic obstacle problem,
\[
\bar{v} := \min \{ \bar{v} \in C(\R^d) : \bar{v} \geq 0,  \bar{F}_{\eta}(D^2 \bar{v}) \leq 0 \mbox{ in $W$,and } \bar{F}_0(D^2 \bar{v}) \leq 0 \mbox{ in $\R^d$}\},
\]
where $\bar{F}_\eta$ is a nonrandom, degenerate elliptic operator defined implicitly at the end 
of Section \ref{subsec:identify},
\[
\begin{split}
\bar{F}_0(M) :=  \inf \{ s \in \R | &\mbox{ there exists $u: \Z^d \to \Z$ such that for all $y \in \Z^d$,} \\
&\mbox{ $\Delta_{\Z^d} u(y) \leq 2 d -1$, and $u(y) \geq \frac{1}{2} y^T(M - s I) y$} \},
\end{split}
\]
and the differential inequalities are interpreted in the viscosity sense. 

 In turn, almost surely, the rescaled sandpiles, $\bar{s}_n(x) := s_n( [n x])$ converge weakly-* to a deterministic function $s \in L^\infty(\R^d)$ as $n \to \infty$. Moreover, the limit satisfies 
$\int_{\R^d} s = |W| \E( \eta(0))$, $s \leq 2 d-1$, $s = 0$ in $\R^d \backslash B_R^{}(W)$ for some constant $R > 0$ depending on $W$ and $\E(\eta(0))$, and weakly, 
\[
s = \begin{cases}
 \Delta \bar{v} + \E( \eta(0))  &\mbox{ in $W$ }  \\
  \Delta \bar{v}  &\mbox{ in $\R^d \backslash W$}.
  \end{cases}
\]

\end{theorem}

\begin{figure}
\begin{centering}
  \includegraphics[width=0.5 \textwidth]{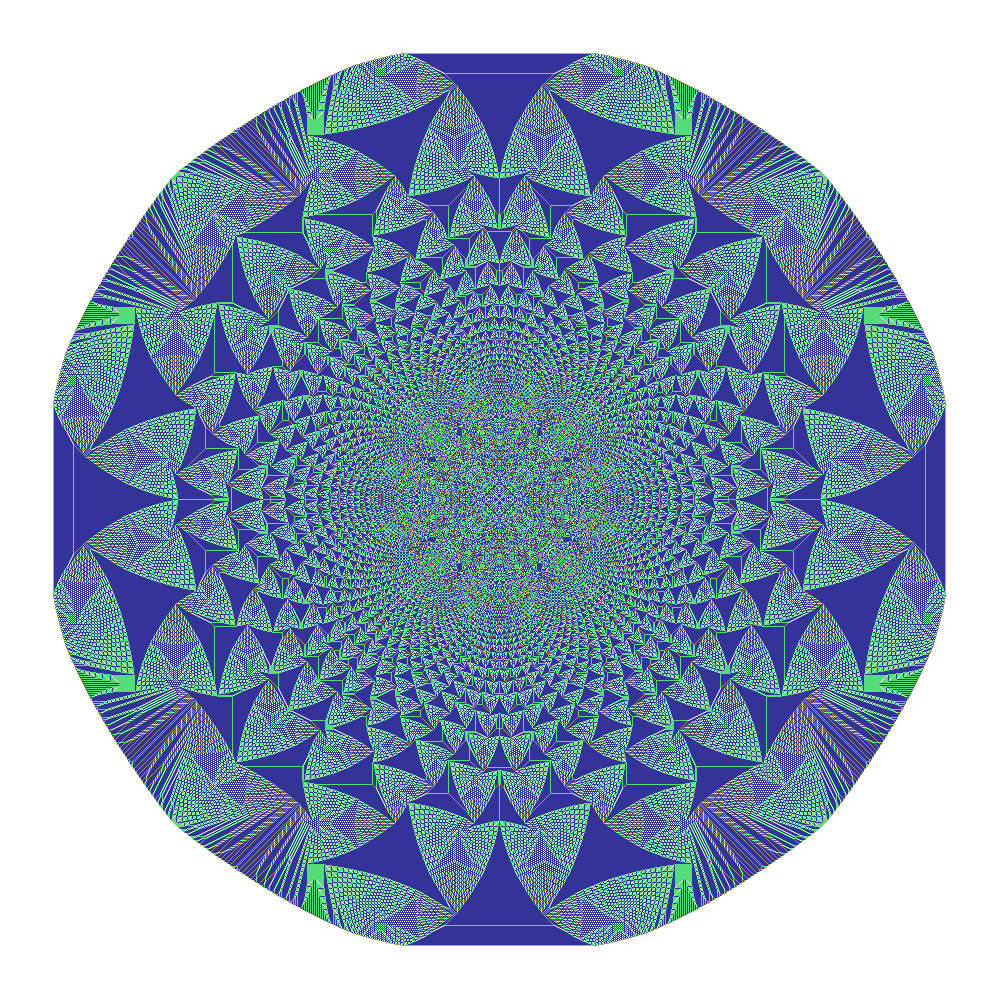}
 \caption{The single source sandpile: start with $10^6$ chips at the origin and stabilize. Sites with 0,1,2, and 3 chips are represented by white, brown, green, and blue respectively.  }
\end{centering}
\end{figure}
	The main challenge in proving the above theorem is that there is no inherent linear or subadditive quantity governing the behavior of the sandpile. The Abelian sandpile is nonlocal: one unstable pile can cause a far-reaching avalanche of topplings. This difficulty is the same one faced by those studying stochastic homogenization of fully nonlinear elliptic PDEs.  Fortunately, since the sandpile can be expressed as the solution to a nonlinear discrete PDE, we can use those same methods here. To be specific, we import the stochastic homogenization tools introduced by S. Armstrong and C. Smart in \cite{armstrong2014quantitative}. 
	
	The tools, however, don't work out of the box. The sandpile doesn't directly fit into the general framework of fully nonlinear elliptic PDEs; and so appropriate sandpile substitutes must be identified in order to run the program. 
Further, a main technical hurdle to overcome is the lack of uniform ellipticity. This is done with new arguments which utilize the regularity theory of the discrete Laplace operator
as well as a comparison principle hidden in sandpile dynamics.

\subsection{Outline of the paper}
In Section \ref{prelims}, we precisely state the assumptions of our result. Then, in Section \ref{sec:sandpile}, we recall some necessary properties of the Abelian sandpile. We provide a brief overview of the main ideas of the proof 
in Section \ref{sec:overview}. Next, in Section \ref{sec:monotone_quantity}, we define a subadditive quantity, $\mu$, which will implicitly control the random sandpile. Through an appropriate perturbation of $\mu$, we identify $\bar{F}_\eta$ in Section \ref{sec:identify}. 
In Section \ref{sec:proof_of_theorem}, we prove the main result by establishing compactness of the odometer 
and showing that $\bar{F}_{\eta}$ has a comparison principle.  Then, in Section \ref{sec:divisible_sandpile} we show a simpler proof of convergence of a related model, the random {\it divisible} sandpile, introduced by L. Levine and Y. Peres  \cite{levine2009strong,levine2010scaling}. In stark contrast to the Abelian sandpile, the limit of the random divisible sandpile is exactly the limit of the averaged divisible sandpile. We end with some generalizations and open questions in Section \ref{sec:conclusion}.


\section{Preliminaries} \label{prelims}

\subsection{Notation and Conventions} 
The constant $d \in \N$ will refer to dimension.
We denote $\symm^d$ as the set of symmetric $d \times d$ matrices with real entries. If $M \in \symm^d$, we write $M \geq 0$ if $M$ has nonnegative eigenvalues.  $|M|_2$ will also refer to the largest in magnitude eigenvalue of $M$. 
 For a vector $x \in \R^d$, $|x|_\infty$ denotes the maximum norm and $|x|_2$ the 2-norm. We sometimes omit the subscript, in which case $|x|$ refers to the 2-norm. We also write $q_M(x) : =  \frac{1}{2} x^T M x$ and $q_l(x) :=  \frac{1}{2d} l |x|^2$. For both functions and vectors, inequalities, minima, and maxima are to be interpreted as pointwise. We write $y \sim x$ when $(y - x) \in \Z$ and $|y-x| = 1$. 
For a subset $A \subset \Z^d$, $d \geq 1$, denote its discrete boundary by 
\[
\partial A = \{ y \in \Z^d \backslash A : \exists x \in A: y \sim x \}. 
\]
and its discrete closure by 
\[
\bar{A} = A \cup \partial A. 
\]
The diameter of $A$ is 
\[
\diam(A) = \max\{ |x-y|_2 : x,y \in A\}.
\]
For $x \in \Z^d$,
\[
Q_n(x) = \{ y \in \Z^d : |x-y|_\infty  < n \}, 
\]
is the cube of radius $n$ centered at $x$:  
and 
\[
B_n(x) = \{y \in \Z^d: |x - y|_2 < n\}, 
\]
is the ball of radius $n$ centered at $x$. For short, $B_n := B_n(0)$, $Q_n := Q_n(0)$.  We will also use $Q_n$ and $B_n$
to refer to cubes and balls on $\R^d$.

We similarly overload notation so that for $A \subset \Z^d$, $|A|$ refers to the number of points in $A$ and for measurable $L \subset \R^d$, $|L|$ is the Lebesgue measure of $L$. For $f: \Z^d \to \R$, we denote its discrete Laplacian by  
\[
\Delta_{\Z^d} f = \sum_{y \sim x} ( f(y) - f(x)), 
\]
and its discrete second-differences by
\[
\Delta_i f = f(x+e_i) + f(x-e_i) - 2 f(x), 
\]
where $\{ \pm e_i\}$  are the $2 d$ coordinate directions in $\Z^d$.   $\Delta$ will refer to the Laplace operator on the continuum.  The convex hull of a set of points $A$ will be denoted 
\[
\conv(A) = \{y \in \R^d:  y = \sum_{i=1}^{d+1} \lambda_i x_i \mbox{ for $x_i \in A$, $\sum_{i=1}^{d+1} \lambda_i  = 1$, $\lambda_i \in [0,1]$}\}.
\]
Throughout the proofs the constant $C$ will denote a positive constant which may change from line to line. When needed, explicit dependence of 
$C$ on other constants will be denoted by, for example, $C_d$.


\subsection{Assumptions}  \label{subsec:assumptions}
We consider the sandpile on the integer lattice $\Z^d$ for $d \geq 2$, (although this assumption is not an essential one). 
Let $\Omega$ denote the set of all bounded functions $\eta: \Z^d \to \Z$.  Endow $\Omega$ with the $\sigma$-algebra $\mathcal{F}$ generated by $\{ \eta \to \eta(x) : x \in \Z^d\}$. We model the randomness by a probability measure $\mathbf{P}$ on $(\Omega, \mathcal{F})$ with the following properties. First,  there exists $\eta_{\min}, \eta_{\max} \in \Z$ so that for every $x \in \Z^d$, 
\begin{equation}
\mbox{Uniform Boundedness: }  \mathbf{P}\left[ \eta_{\min} \leq \eta(x) \leq \eta_{\max} \right] = 1. 
\end{equation} 
Note this may be replaced by an appropriate concentration assumption. We further assume that $\mathbf{P}$ is \mbox{stationary} and \mbox{ergodic}. Denote the action of translation by $T: \Z^d \times \Omega \to \Omega$, 
\[
T(y, \eta)(z) = (T_y \eta)(z) := \eta(y + z),
\]
and extend this to $\mathcal{F}$ by defining  $T_y E := \{ T_y \eta : \eta \in E\}$. Stationarity and ergodicity is then
\begin{equation}
\mbox{Stationarity: for all $E \in \mathcal{F}, y \in \Z^d$: } \mathbf{P}(T_y E) = \mathbf{P}(E),
\end{equation}
\begin{equation}
\mbox{Ergodic:  $E = \cap_{y \in \Z^d} T_y E$ implies that $\mathbf{P}(E) \in \{0,1\}$ } . 
\end{equation}
Lastly, we assume that the density of sand in the initial sandpile is high:
\begin{equation}
\mbox{ High density: }  \E(\eta(0)) > 2 d - 1. 
\end{equation}
High density is a natural, weak assumption which forces interesting behavior to occur. See Section \ref{sec:conclusion} for further discussion of this assumption. A concrete example to keep in mind 
is when $\{\eta(x)\}_{x \in \Z^d}$ are independent and identically distributed with sufficiently large expectation.

\section{Sandpiles}\label{sec:sandpile}
The results in this section are reformulations of fundamental facts about the Abelian sandpile. See, for example, \cite{corry2018divisors,klivans2018mathematics,jarai2018sandpile,holroyd2008chip, redig2005mathematical}. Fix a bounded, connected $A \subset \Z^d$ and a starting sandpile $\eta: A \to \Z$.  We call positive integer-valued functions on $\bar{A}$,  {\it toppling functions}.  Recall that a toppling function $u$ is {\it legal} for $\eta$ if it can be decomposed into a sequence of topplings so that only sites $x$ where $\eta(x) \geq 2 d$ are toppled. More precisely, $u$ is legal for $\eta$ if we can express for some $n \geq 0 $,
\[
u = u_0 +  u_1 + \cdots + u_n, 
\]
where $u_0 = 0$ and for $i \geq 1$, $u_i(x) = 0$ for all $x \in A$ except for one $\hat{x}_i \in A$ for which $u_i(\hat{x}_i) = 1$
and 
\[
\left( \Delta_{\Z^d} (u_1 + \cdots +  u_{i-1})  + \eta \right)(\hat{x}_i) \geq 2 d.
\]
When $\eta \leq 2d -1$, the only legal toppling function is $u = u_0 = 0$. An important observation is that any legal toppling function satisfies $\Delta_{\Z^d} u + \eta \geq \min(0, \eta)$ but this inequality does not imply $u$ is legal. A toppling function $v$ is {\it stabilizing} for $\eta$ in $A$ if $\Delta_{\Z^d} v + \eta \leq 2 d-1$ in $A$. 

Denote the set of {\it locally legal} topplings for $\eta$ as
\begin{align*}
\mathcal{L}(\eta, A) = &\{ u: \bar{A} \to \N :  \mbox{ there exists $w:\bar{A} \to \N$ and $\hat{u}: \bar{A} \to \N$} \\
&\mbox{where $u = w + \hat{u}$, $w(x) = 0$ for $x \in A$, $\hat{u}(x) = 0$ for $x \in \partial A$}  \\
&\mbox{and $\hat{u}$ is legal for $\Delta_{\Z^d} w + \eta$ in $A$}\}
\end{align*}
and the set of stabilizing topplings for $\eta$ in $A$ as 
\[
\mathcal{S}(\eta, A)  = \{ v: \bar{A} \to \N:   \Delta_{\Z^d} v(x) + \eta(x)  \leq 2 d-1 \mbox{ for $x \in A$} \}.
\]
It is important to note that these sets only enforce their constraints in $A$, they may include arbitrary topplings on $\partial A$.

The {\it odometer} function, $v: \Z^d \to \N$, counts the number of times each site in $\eta$ topples when stabilizing. Here we distinguish between two common scenarios. 
In the first scenario, once a grain leaves $A$, it falls off and disappears. We call this the {\it open boundary} condition. In this case $v = 0$ on $\Z^d \backslash A$. In the second scenario, 
grains continue to spread and topplings can occur outside of $A$. This is the {\it free boundary} condition. The sandpile we consider in our main theorem has the free 
boundary condition. However, as we will discuss in Section \ref{sec:conclusion}, our methods also apply to other sandpiles including those with open boundaries. In this section, 
we state results for sandpiles with open boundaries. 

First, we recall the least-action principle for sandpiles \cite{fey2010growth} and rephrase it in a way amenable to the methods of this paper. We will refer to this as the {\it discrete sandpile PDE}. 
\begin{prop}\label{prop:least_action}
The odometer function $v$ uniquely solves each of the following problems. 
\begin{enumerate}
\item  Longest legal toppling,
\[
v = \sup\{ w: \bar{A} \to \N : w \in \mathcal{L}(\eta, A), \mbox{ $w = 0$  on $\partial A$}\}.
\]
\item Shortest stabilizing toppling, 
\[
v = \inf\{ u: \bar{A} \to \N : u \in \mathcal{S}(\eta, A),  \mbox{ $u = 0$  on $\partial A$}\}. 
\]
\item Stabilizing, legal toppling, 
\[
 v \in \mathcal{L}(\eta, A) \cap \mathcal{S}(\eta, A) \mbox{ and $v = 0$ on $\partial A$}.
 \]
\end{enumerate}
\end{prop}
The reader should view locally legal toppling functions as subsolutions and stabilizing toppling functions as supersolutions.

We will also use the following consequence of the Abelian property:  any locally legal, stabilizing toppling function 
can be decomposed into the usual odometer function for $\eta$ and an odometer function
which keeps track of topplings originating from the boundary. 
\begin{prop} \label{Abelian}
	If $v \in \mathcal{L}(\eta,A ) \cap \mathcal{S}(\eta, A)$ and $v = f \geq 0$ on $\partial A$, then $v$ can be decomposed into
	\[
	v = v_1 + v_2, 
	\]
	where
	\[
	\begin{cases}
	v_1 \in \mathcal{L}(\eta) \cap \mathcal{S}(\eta)  &\mbox{ on $A$} \\
	v_1 = 0 & \mbox{on $\partial A$} 
	\end{cases}
	\]
	and
	\[
	\begin{cases}
	v_2 \in \mathcal{L}(\eta + \Delta_{\Z^d} v_1) \cap \mathcal{S}(\eta + \Delta_{\Z^d} v_1)  &\mbox{ on $A$} \\
	v _2= f & \mbox{on $\partial A$} 
	\end{cases}
	\]
\end{prop}

A certain class of sandpiles, known as {\it recurrent} sandpiles will help in the sequel. We say $\eta$ is recurrent if we can find $s: A \to \N$ and $u \in \mathcal{L}(s + 2 d - 1, A)$ with $u = 0$ on $\partial A$ so that $\eta = 2 d- 1 + s + \Delta_{\Z^d} u$ in $A$. In other words, $\eta$ is recurrent if we can reach $\eta$ by starting with $2d-1$ chips at every site in $A$, adding chips at some sites in $A$, and then toppling some sites legally. Also we call $\eta$  {\it stable} in $A$ if $\eta \leq 2 d-1$ in $A$.  

A useful consequence of Dhar's burning algorithm \cite{dhar1990self} will aid in controlling topplings in stable, recurrent sandpiles.  Recall that the burning algorithm provides a recipe for checking if a stable sandpile is recurrent:
topple the boundary of a sandpile once, if the sandpile is recurrent, each inner site will topple exactly once when stabilizing. More generally, topple sites along $\partial A$ and then legally stabilize $s$ in $A$. If $s$ is a stable sandpile, no site in $A$ will topple more times than a boundary site has toppled. And, if $s$ is a recurrent sandpile, every site in $A$ will topple at least as many times as some boundary site. This leads to both a maximum principle and a comparison principle for the sandpile.

\begin{prop}\label{burning}
For $f: \partial A \to \N$, a sandpile $s: A \to \Z$, let $v$ solve
\[
\begin{cases}
v \in \mathcal{L}(s) \cap \mathcal{S}(s)  &\mbox{ on $A$} \\
v = f & \mbox{on $\partial A$} 
\end{cases}
\]
If $s$ is stable, then
\begin{equation} \label{eq:stable}
\sup_{x \in A} v(x) \leq \sup_{y \in \partial A} f(y). 
\end{equation}
If $s$ is recurrent, then
\begin{equation} \label{eq:rec}
\inf_{x \in A} v(x) \geq \inf_{y \in \partial A} f(y). 
\end{equation}
In particular, when $s$ is stable and recurrent, we have the following comparison principle: 
let $u$ solve
\[
\begin{cases}
u \in \mathcal{L}(s) \cap \mathcal{S}(s)  &\mbox{ on $A$} \\
u = f' & \mbox{on $\partial A$}, 
\end{cases}
\]
for some $f': \partial A \to \N$.
Then, for any integer-valued harmonic functions $g, h : \bar{A} \to \Z$ with $\Delta_{\Z^d} g = \Delta_{\Z^d} h = 0$ in $A$, 
\begin{equation} \label{eq:comp}
\inf_{x \in A} ( (u + g) - (v + h) )(x) \geq \inf_{y \in \partial A} ( (f'+g) - (f + h))(y). 
\end{equation}
\end{prop}

\begin{proof}
	The maximum principle, \eqref{eq:stable} and \eqref{eq:rec}, follows from the proof of the burning algorithm, see Theorem 2.6.3 in \cite{klivans2018mathematics} or Theorem 7.5 in \cite{corry2018divisors}.  We show how these imply \eqref{eq:comp}.  By definition, $\eta_u := \Delta_{\Z^d} u + s$ and $\eta_v := \Delta_{\Z^d} v + s$ 
	are both stable and recurrent. Let 
	\[
	w(x) =  ( (u + g) - (v + h) )(x) - \inf_{y \in \bar{A}} ( (u + g) - (v + h) )(y), 
	\]
	so that $w \geq 0$ and in $A$, 
	\[
	\Delta_{\Z^d} w + \eta_v = \Delta_{\Z^d} u + s = \eta_u. 
	\]
	Let $\hat{w}$ solve 
	\[
 	\hat{w} \in \mathcal{L}(\eta_v, A) \cap \mathcal{S}(\eta_v, A) \mbox{ and $\hat{w} = w$ on $\partial A$}.
	\]
	By Propositions \ref{Abelian}, \ref{prop:least_action}, and then \eqref{eq:rec}, 
	\[
	\inf_{x \in A} w(x) \geq \inf_{x \in A} \hat{w}(x) \geq \inf_{y \in \partial A} \hat{w}(y) = \inf_{y \in \partial A} w(y).
	\]
\end{proof}

We conclude the section by noting a useful alternative characterization of recurrent sandpiles.  If each site in $\eta$
has toppled at least once, then what remains is recurrent. 
\begin{prop}\label{recurrent}
Let $A \subset W$ be connected subsets of $\Z^d$.  If $w \in \mathcal{L}(\eta, W)$ with $w \geq 1$ in $A$, then $\Delta_{\Z^d} w+ \eta$ is recurrent in $A$. 
\end{prop}

\section{Method overview} \label{sec:overview}
Let $W \subset \R^d$ be a bounded Lipschitz domain and $\eta :\Z^d \to \Z$ a uniformly bounded initial sandpile. 
For every $n \geq 1$ and finite difference approximation $W_n := n W \cap \Z^d$, there is a unique odometer function $v_n : W_n \to \N$ which determines the stabilization of $\eta$ in $W_n$. Uniqueness of $v_n$ comes from the fact that it solves the discrete sandpile PDE for $\eta$ on $W_n$ and this PDE enjoys a comparison principle, Proposition \ref{burning}. 

Hence, for each fixed $n$, $v_n$ is unique. As $\eta$ is uniformly bounded, the Laplacians of $v_n$
are in turn bounded and $\bar{v}_n(x):= n^{-2} v_n([n x])$ converges along {\it subsequences} uniformly in $W$. However, there is no reason to expect the limits to coincide for arbitrary choices of $\eta$. We must have, at least, 
convergence of local averages of $\eta$ itself. A natural choice is to assume that $\eta$ is drawn from a stationary and ergodic probability distribution. 

If we remove the integer constraint on the sandpile PDE, we can use simple random walk 
difference estimates (see Section \ref{sec:divisible_sandpile}), to show that the limit {\it divisible} sandpile
coincides with the limit of the averaged divisible sandpile, $\E(\eta(0))$. In this case, the odometers 
solve a linear PDE in the limit, the Poisson problem, $\Delta \bar{v} + \E(\eta(0)) = 2 d -1$, on $W$. 

The integer constraint imposes a nonlinear structure on the problem which makes it more difficult. 
Essentially our only available tool is a discrete comparison principle. Our method can be understood 
as a technique to push the comparison principle for the sandpile from the lattice to the continuum. We show, roughly,
that the discrete sandpile PDE converges to a fully nonlinear elliptic PDE with a comparison principle.
The limit PDE inherits features from both the lattice $\Z^d$ and the distribution of $\eta$. (Interestingly, simulations 
suggest that the limit PDE is wildly different for different probability distributions, even when they share the same mean.)

Our proof of Theorem \ref{thetheorem} follows the stochastic homogenization program of Armstrong and Smart \cite{armstrong2014quantitative}. The strategy involves comparing the odometer function of the sandpile to the solution of an auxillary Monge-Amp{\`e}re equation \cite{gutierrez2016monge,trudinger2008monge} which will control how much the odometer function can `bend'. The proof has three steps: 
\begin{enumerate}
	\item Identify the {\it effective equation} $\bar{F}_{\eta}$ describing the limit PDE via the subadditive ergodic theorem applied to an auxiliary Monge-Amp{\`e}re measure, $\mu$.
	\item Show that $\bar{F}_{\eta}$ inherits the comparison principle for sandpiles.
	\item Conclude by showing every subsequential limit is a viscosity solution to $\bar{F}_{\eta}(D^2(v)) = 0$.
\end{enumerate}

The most difficult part of this program is to show that $\bar{F}_{\eta}$ has a comparison principle, Lemma \ref{flat}. In the fully nonlinear elliptic setting, this is done under the assumption of {\it uniform ellipticity}. The discrete sandpile PDE is not uniformly elliptic (a priori), so our argument for this is completely new. Lack of uniform ellipticity also required us to develop new arguments for the regularity of $\mu$, Lemma \ref{strict-convexity} and Lemma \ref{continuity}.

\section{A monotone quantity} \label{sec:monotone_quantity}

\subsection{The definition of $\mu$} 
In this section we introduce $\mu$, a monotone quantity which will control solutions to the discrete sandpile PDE. 
For a function $v: \Z^d \to \Z$ and $x \in A \subset \Z^d$ let
\[
\partial^+(v, x,A)  = \{ p \in \R^d : v(x) + p \cdot (y-x) \geq v(y) : \mbox{ for all $y \in \bar{A}$}\}
\]
denote the supergradient set of $v$ at $x$ in $A$. Similarly, 
\[
\partial^-(v, x,A)  = \{ p \in \R^d : v(x) + p \cdot (y-x) \leq v(y) : \mbox{ for all $y \in \bar{A}$}\}
\]
is the subgradient set at $x$.  For short-hand, we omit the set $A$ when it is clear and write
\[
\partial^+(v, A) =  \cup_{x \in A} \partial^+(v,x). 
\]

To completely identify a fully nonlinear elliptic PDE, it suffices to recognize when a parabola is a supersolution or a subsolution. This fundamental observation is due to Caffarelli \cite{caffarelli1999note} and was employed by Caffarelli, Souganidis, and Wang in their obstacle problem argument for stochastic homogenization of fully nonlinear uniformly elliptic equations \cite{caffarelli2005homogenization, armstrong2014stochastic}.

Our method is similar: we will perturb solutions by a parabola and then define the effective equation, $\bar{F}_{\eta}$, through these perturbed limits. For $l \in \R$, $M \in \symm^d$, and $\eta \in \Omega$, denote the set of perturbed subsolutions as
\[
S(A, \eta, l, M) =  \{ u: \bar{A} \to \N : u \in \mathcal{L}(\eta, A) \}  - \left( q_l  + q_M \right)
\] 
and the set of perturbed supersolutions as
\[
S^*(A, \eta, l, M) =   \{ v: \bar{A} \to \N : v \in \mathcal{S}(\eta, A)  \}  - \left( q_l + q_M \right). 
\]
The monotone quantity controlling subsolutions is then
\[
\mu(A, \eta, l, M) = \sup \{ |\partial^+(w,A)|:   w \in S(A, \eta, l, M) \},
\]
while the monotone quantity which controls supersolutions is
\[
\mu^*(A, \eta, l, M)  = \sup \{ |\partial^-(w,A)|: w \in S^*(A, \eta, l , M)  \}. 
\]
The results in the next two subsections are completely deterministic, so we fix $\eta \in \Omega$.

\subsection{Comparing subsolutions and supersolutions}

We will need to compare legal and stabilizing toppling functions throughout this paper. However, the discrete sandpile PDE is nonlinear:  
if $v$ is a stabilizing toppling function for $\eta$, then $-v$ is not a legal toppling function for $-\eta$ (unless $v = 0$).  This makes 
it difficult to compare legal and stabilizing toppling functions. However, through $\mu$, we can compare the two using the following lemma, which roughly states that legal, stabilizing 
toppling functions maximize curvature. 

\begin{lemma} \label{stabilizing-legal}
If $u \in \mathcal{L}(\eta, A)$, the solution of
\[
h \in \mathcal{L}(\eta, A) \cap  \mathcal{S}(\eta, A) \mbox{ and $h = u$ on $\partial A$}, 
\]
satisfies $\partial^+(u, A) \subseteq \partial^+(h, A)$. Similarly, if $v \in \mathcal{S}(\eta, A)$, then the solution of 
 \[
h^* \in \mathcal{L}(\eta, A) \cap  \mathcal{S}(\eta, A) \mbox{ and $h^* = v$ on $\partial A$}, 
 \]
 satisfies $\partial^-(v, A) \subseteq \partial^-(h^*, A)$.
 \end{lemma}

\begin{proof}
Take $p \in \partial^+(u,x,A)$ and let
\[
t = \inf \{ c \in \R: u(x) + p \cdot(y-x) + c \geq h(y) \mbox{ for all $y \in \bar{A}$} \}
\]
By the least action principle and boundary assumption, $h \geq u$, hence $t \geq 0$. 
Also, as $A$ is finite, $t < \infty$. Since $h= u$ on $\partial A$, 
we must have $y \in A$ for which 
\[
u(x) + p \cdot(y-x) + t = h(y) , 
\]
which shows $p \in \partial^+(h, y, A)$. The proof for subgradients is similar. 
\end{proof}

\subsection{Basic properties of $\mu$}

We now establish control on solutions from above and below which will follow from the proof of the Alexandroff-Bakelman-Pucci (ABP) inequality (Theorem 3.2 in \cite{roberts1995fully} and Theorem 1.4.2 in \cite{gutierrez2016monge}). 
\begin{lemma} \label{ABP}
There exists $C_d > 0$ so that for all  $w \in  S(B_n, \eta, l , M)$, 
\begin{equation}
\max_{x \in  B_n} w(x) \leq  \max_{x \in \partial B_n} w(x) + C_d n \mu(B_n, \eta, l ,M)^{1/d}
\end{equation}
and for all  $w \in  S^*(B_n, \eta, l , M)$, 
\begin{equation}
\inf_{x \in  \partial B_n} w(x) \leq  \inf_{x \in B_n} w(x) + C_d n \mu^*(B_n, \eta, l ,M)^{1/d}.
\end{equation}
\end{lemma}
\begin{proof}
Let $a = \max_{x \in B_n} w(x) - \max_{x \in \partial B_n} w(x)$. Assume $a > 0$, otherwise the claim is immediate. 
Choose $x_0$ so that $\max_{x \in  B_n} w(x) = w(x_0)$. Let $p \in \R^d$  satisfy $|p| \leq a \mbox{diam}(B_n)^{-1} = C_d a/n$. Then, 
for each $x \in B_n$, 
\begin{equation}\label{ABPeq:1}
\begin{aligned}
w(x_0) + p \cdot (x-x_0) &\geq w(x_0) - |p| |x-x_0| \\
&> w(x_0) - w(x_0) + \max_{x \in \partial B_n} w(x)  \\
&= \max_{x \in \partial B_n} w(x).
\end{aligned}
\end{equation}
Now, we shift the hyperplane up just enough so that it lies above $w$ in $\bar{B}_n$: let 
\[
t = \inf\{ c  \in \R:  w(x_0) + p \cdot(x-x_0)  + c \geq w(x) \mbox{ for all $x \in \bar{B}_n$} \}
\]
and note that $t \geq 0$ and that there exists $y \in \bar{B}_n$ with 
\[
w(y) = w(x_0) + p \cdot(y-x_0)  + t. 
\]
If $t > 0$, then \eqref{ABPeq:1} shows that $y \in B_n$.   If $t = 0$, we can choose $y = x_0$. Hence, there is a $y \in B_n$ with $p \in \partial^+(w, y , B_n)$. Since this holds for every $|p|  < a/\mbox{diam}(B_n)$, 
this implies
\[
|\partial^+(w, B_n)| \geq C_d \frac{a^d}{ {\diam}(B_n)^d}.
\]
And so rearranging, we get 
\[
a \leq |\partial^+(w, B_n)|^{1/d} C_d \diam(B_n) \leq C_d n \mu(B_n, \eta, l, M)^{1/d} 
\]
The proof for $\mu^*$ is identical.

\end{proof}

Next we introduce the concave envelope of a subsolution. First, we extend the discrete domain $Q_n$ and its closure to their convex hulls: $\mathcal{Q}_n := \conv{Q_n}$ and
$\mathcal{\bar{Q}}_n := \conv{\bar{Q}}_n$. 
Then, define the concave envelope of $w$ by, $\Gamma_w:  \mathcal{\bar{Q}}_n \to \R$, 
\[
\Gamma_w(x) = \inf_{p \in \R^d} \max_{y \in \bar{Q}_n} \left( w(y) + p \cdot (x-y) \right), 
\]
noting that $\Gamma_w$ is the pointwise least concave function so that on $\bar{Q}_n$, $\Gamma_w \geq w$. We recall a useful representation
of the concave envelope.

\begin{prop}[Lemma 4.5 in \cite{imbert2013introduction}] \label{prop:concave_rep}
We can alternatively represent
\[
\Gamma_w(x) = \sup\{ \sum_{i=1}^{d+1} \lambda_i w(x_i) : x_i \in  \bar{Q}_n, \sum_{i=1}^{d+1} \lambda_i x_i = x,  \lambda_i \in [0,1], \sum_{i=1}^{d+1} \lambda_i = 1\}, 
\]
and if 
\[
\Gamma_w(x) = \sum_{i=1}^{d+1} \lambda_i w(x_i) , 
\]
then for each $x_i$, $\Gamma_w(x_i) = w(x_i)$ and $\Gamma_w$ is linear in $\conv(x_1, \ldots, x_{d+1})$. 
\end{prop}

The next statement uses this representation to show that the measure of the supergradient set is preserved under the operation of taking the concave envelope. As the concave envelope is defined on $\R^d$, we first extend the definition of supergradient set
to functions $g: \mathcal{Q}_n \to \R$, 
\begin{equation}
\partial^+(g, \mathcal{Q}_n) = \{p \in \R^d: \exists x \in \mathcal{Q}_n : g(x) + p \cdot (y-x) \geq g(y) : \mbox{ for all $y \in \mathcal{\bar{Q}}_n$}\}.
\end{equation}

\begin{lemma} \label{MA-rep} 
\[
\begin{aligned}
&\sum_{x \in Q_n} |\partial^+(w, x, Q_n)| = |\partial^+(w, Q_n)| \\
&=  \sum_{\{ x: \Gamma_w(x) = w(x) \}}  |\partial^+(\Gamma_w, x, Q_n)| = |\partial^+(\Gamma_w, \mathcal{Q}_n)|. 
\end{aligned}
\]
\end{lemma}

\begin{proof}
We split the proof into two steps.
\subsubsection*{Step 1}
We first show that 
\[
|\partial^+(w, Q_n)| = \sum_{x \in Q_n} |\partial^+(w, x, Q_n)|, 
\]
which follows from the proof in the continuous setting:
since 
\[
|\partial^+(w, Q_n)| = | \cup_{x \in Q_n} \partial^+(w, x)|, 
\]
it suffices to show that
\[
S = \{ p \in \R^d : \mbox{ there exists $x,y \in Q_n$, $x \not = y$ and $p \in \partial^+(w, x) \cap \partial^+(w, y)$} \} 
\]
has measure zero.  Denote the discrete Legendre transform $w^*: \R^d \to \R$ by  $w^*(p) := \min_{x \in \bar{Q}_n} (x \cdot p - w(x))$. This is a concave, finite function as $Q_n$ is bounded
and it is a minimum of affine functions.  Further,  if $p \in \partial^+(w, x)$, then $w^*(p) = x \cdot p - w(x)$. Hence, if $p \in S$ then $w^*(p) = x_1 \cdot p - w(x_1) = x_2 \cdot p - w(x_2)$ for $x_1 \not = x_2$. This implies that $w^*(p)$ is not differentiable at $p$. But, since $w^*$ is concave it is differentiable almost everywhere, which implies $S$ has measure zero since it is a subset of a measure zero set. This completes the proof of Step 1.

\subsubsection*{Step 2} 
We now show that 
\[
|\partial^+(w, Q_n)| = |\partial^+(\Gamma_w, \mathcal{Q}_n)|  = \sum_{\{ x: \Gamma_w(x) = w(x) \}}  |\partial^+(\Gamma_w, x, Q_n)|.
\]
First consider $p \in \partial^+(w, x)$ and the affine function $L(y) = w(x) + p \cdot(y-x)$ for $y \in \mathcal{Q}_n$. By definition of the concave envelope, for any $y \in \mathcal{Q}_n$, 
\[
\Gamma_w(x) + p \cdot (y-x) \geq w(x) + p \cdot(y-x) = L(y) \geq \Gamma_w(y),
\]
and so $p \in \partial^+(\Gamma_w, \mathcal{Q}_n)$.

Next, take $p \in \partial^+(\Gamma_w, x)$ for $x \in \mathcal{Q}_n$ and use Proposition \ref{prop:concave_rep} to express 
\[
\Gamma_w(x)  = \sum_{i=1}^{k} \lambda_i w(x_i), 
\] 
for $\lambda_i > 0$, $x_i \in \bar{Q}_n$, and some $k \geq 1$. This implies that $p \in \partial^+(\Gamma_w, x_i)$ for some $x_i \in \bar{Q}_n$. If $k = 1$ and $x_i = x \in Q_n$, we are done as $\Gamma_w(x_i) = w(x_i)$, so suppose not.  Then, we can find some $x_i \not = x$ and $p \in \partial^+(\Gamma_w, x) \cap \partial^+(\Gamma_w, x_i)$. 
However, the argument in Step 1 implies that such $p$ have measure zero. This also implies the third equality.

\end{proof}

The arithmetic-geometric mean inequality and the lower bound on the Laplacian of subsolutions imply an upper bound on $\mu$. 
\begin{lemma} \label{mu-is-bounded}
There is $C := C_{ \eta_{\max}, l, M, d}$ and $C^* := C_{\eta_{\min}, l, M, d}^*$  for which 
\[
\mu(Q_n, \eta, l ,M) < C |Q_n| ,
\]
\[
\mu^*(Q_n, \eta, l ,M) < C^* |Q_n|.
\]
For $l \leq -\eta_{\max} - \mbox{Tr}(M)$
\[
\mu(Q_n, \eta, l , M) = 0 
\]
and for $l \geq (2 d-1) - \eta_{\min} - \mbox{Tr}(M)$
\[
\mu^*(Q_n, \eta, l, M) = 0. 
\]

\end{lemma}

\begin{proof}

Let $w := u - q_l - q_M \in S(A, \eta, l, M)$. Since $u$ is legal,  $\Delta_{\Z^d} u \geq \min(-\eta_{\max},0)$ in $Q_n$. 
Using $\Delta_{\Z^d} q_M = \mbox{Tr}(M)$, we get $\Delta_{\Z^d} w \geq  - l - \mbox{Tr}(M) -\eta_{\max}$.  

Choose $x \in Q_n$ so that  $|\partial^+(w, x)| > 0$. As the supergradient set is preserved under affine transformations, we may suppose $w(x) = 0$ and $0 \in \partial^+(w, x)$. 
 This implies $w(y) \leq 0$ for all $y \in \bar{A}$. Then, by definition, for $p \in \partial^+(w,x)$, 
\[
p \cdot(x + e_i - x)  \geq w(x+e_i),
\]
and
\[
p \cdot( x - e_i - x) \geq w(x-e_i). 
\]
Putting these two inequalities together, we get for each direction $i = 1, \ldots, d$, 
\begin{equation} \label{discreteMA1}
w(x+e_i)  \leq p_i \leq - w(x-e_i). 
\end{equation}
And so, 
\[
|\partial^+(w,x)| \leq \prod_{i=1}^d \left(- w(x-e_i) - w(x+e_i) \right) = \prod_{i=1}^d (- \Delta_i w).
\]
Our affine transformation of $w$ ensures that $-\Delta_i w \geq 0$, and so an application of the arithmetic geometric mean inequality yields
\[
- \Delta_{\Z^d} w = \sum_{i=1}^d (-\Delta_i w) \geq d \left( \prod_{i=1}^d (-\Delta_i w) \right)^{1/d}.
\]
And so
\begin{equation}
|\partial^+(w,x)|  \leq  d^{-d} (-\Delta_{\Z^d} w)^d \leq  d^{-d} (\eta_{\max} + \mbox{Tr}(M) + l)^d, 
\end{equation}
which implies the claim by Lemma \ref{MA-rep}. The other direction is similar.


\end{proof}

We state the following consequence of the discrete Harnack inequality \cite{lawler2010random} which we will later use to regulate the growth of the concave envelope in balls around contact points. 

\begin{prop}[Lemma 2.17 in \cite{levine2010scaling}] \label{harnack}
Fix $0 < \beta < 1$. For any $f: \Z^d \to \R$ nonnegative, with $f(0) = 0$ and $|\Delta_{\Z^d} f| \leq \lambda$ in $B_R$  there is a constant $C_{\beta,\lambda}$ so that
\begin{equation} \label{harnack_eq}
f(x) \leq C_{\beta,\lambda}  |x|^2
\end{equation}
for $x \in B_{\beta R}$.
\end{prop}

\subsection{Convergence of $\mu$}

We next use the multiparameter subadditive ergodic theorem of Akcoglu and Krengel \cite{akcoglu1981ergodic}
as modified by Dal Maso and Modica \cite{dal1985nonlinear} to show almost sure convergence of $\mu$. 
For the reader's convenience, we restate the theorem, following 
the exposition in \cite{armstrong2014regularity}. 

Let $\mathcal{U}_0$ be the family of bounded subsets of $\Z^d$ and $\mathcal{L}$ the set 
of bounded Lipschitz domains in $\R^d$. A function $f: \mathcal{U}_0 \to \R$ 
is {\it subadditive} if 
\[
f(A) \leq \sum_{j=1}^{k} f(A_j), 
\]
whenever $k \in \N$ and $A, A_1, \ldots, A_k \in \mathcal{U}_0$ are such that $A_1, \ldots, A_k$ are pairwise disjoint and $A = \cup_{j=1}^k A_j$. For a fixed constant $C$, let $\mathcal{M}_C$
be the collection of subadditive functions $f: \mathcal{U}_0 \to \R$ which satisfy 
\[
0 \leq f(A) \leq C |A| \mbox { \qquad for every $A \in \mathcal{U}_0$}.
\]
A {\it subadditive process} is a function $f: \Omega \to \mathcal{M}_C$. Overload notation and write 
$f(A, \eta) = f(\eta)(A)$ for $A \in \mathcal{U}_0$ and $\eta \in \Omega$. Recall that we have assumed the probability measure is stationary and ergodic.

\begin{prop}[Multiparameter subadditive ergodic theorem] \label{prop:subadditive}
	Let $f: \Omega \to \mathcal{M}_C$ be a subadditive process. There exists
	an event $\Omega_0$ of full probability and a constant $0 \leq a \leq C$ so
	that for every $\eta \in \Omega_0$ and $W \in \mathcal{L}$, 
	\begin{equation}
	\lim_{n \to \infty} \frac{f(n W \cap \Z^d, \eta)}{|n W \cap \Z^d|} = a. 
	\end{equation}
	
\end{prop}

The next lemma is an easy consequence.


\begin{lemma} \label{sub-additive-ergodic}
For each $M \in \symm^d$ and $l \in \R$, there exists an event, $\Omega_{l,M}$, of full probability so that 
for every $\eta \in \Omega_{l,M}$ and $W \in \mathcal{L}$, 
\[
\mu(l, M): = \lim_{n \to \infty} \frac{\mu(n W \cap \Z^d, \eta, l, M) }{|n W \cap \Z^d|}
\]
and
\[
\mu^*(l,M) := \lim_{n \to \infty} \frac{\mu^*(n W \cap \Z^d , \eta, l, M) }{|n W \cap \Z^d|}.
\]
Moreover, there exist constants $C:=C_{\eta_{\max},l,M,d}$ and $C^*:=C^*_{\eta_{\min},l,M,d}$ so that
\[
0 \leq \mu(l,M) \leq C
\]
and
\[
0 \leq \mu(l,M)^* \leq C^*.
\]

\end{lemma}

\begin{proof}
Fix $M$ and $l$ and let $W_n = n W \cap \Z^d$ for given $W \in \mathcal{L}$. We apply Proposition \ref{prop:subadditive} to
\[
f(W_n, \eta) = \sup\{ |\partial^+(w, W_n)| : w \in S(W_n, \eta, l, M)\}. 
\]
Let $\Omega_{l,M}$ be given by Proposition \ref{prop:subadditive} and take $\eta \in \Omega_{l,M}$. 
By Lemma \ref{mu-is-bounded}, $0 \leq f(W_n, \eta) \leq  C |W_n|$.  It remains to check subadditivity for subsets of $\Z^d$.  
Let $A \in \mathcal{U}_0$  and let $A_1, \ldots, A_k$ be pairwise disjoint subsets of $A$ which satisfy $\cup_{j=1}^k A_j = A$.

Let $u$ be a locally legal toppling function for $\eta$ in $A$. For each $A_i$, we can decompose $u$ into illegal topplings on $\partial A_i$ followed by locally legal topplings 
in $A_i$. Hence $u - q_l - q_M \in S(A_i, \eta, l,M)$ for all $A_i$. Moreover, the definition of supergradient
shows that for each $x \in A$ there is an $A_i$ so that
\[
\partial^+(u-q_l-q_M, x, A) \subset \partial^+(u - q_l-q_M, x, A_i),
\]
hence by Lemma \ref{MA-rep} and disjointness of the $A_i$, 
\begin{align*}
|\partial^+(u-q_l-q_M, A)| &= \sum_{x \in A} |\partial^+(u-q_l-q_M, x, A)|  \\
&\leq  \sum_{i=1}^k \sum_{x \in A_i} |\partial^+(u - q_l-q_M, A_i, x)| \\
&= \sum_{i=1}^k  |\partial^+(u - q_l-q_M, A_i)|. 
\end{align*}
Since this holds for any locally legal toppling of $\eta$ in $A$, taking the supremum of both sides implies that 
\[
f(A, \eta) \leq \sum_{i=1}^k f(A_i, \eta), 
\]
which completes the proof. The exact same argument, using the fact that any stabilizing toppling for $A$ is also stabilizing in $A_i$, shows convergence of $\mu^*$.
\end{proof}

In light of Proposition \ref{ABP}, if both $\mu(l,M)$ and $\mu^*(l,M)$ are 0, we have a comparison principle in the limit. This will allow 
us to identify the effective equation; and hence is what we carry out in the next section.

\section{The effective equation} \label{sec:identify}

\subsection{Finding the effective equation}  \label{subsec:identify}
We will identify, for each parabola $M$, the largest real number $l_M$, so that in the limit 
$\mu(l_M,M) = \mu^*(l_M,M) = 0$. This then defines the effective equation $\bar{F}_{\eta}$. To show that such a number 
exists, since $\mu$ is bounded, it suffices to show that $\mu$ is Lipschitz continuous in the limit. In the continuum, 
this is done with an argument that utilizes a certain regularity of concave envelopes of subsolutions which we do not have. 
This difficulty is circumvented by a consequence of the stationarity of $\eta$, Lemma \ref{strict-convexity}. 
We first prove the easier direction of continuity, monotonicity of the curvature. 
 \begin{lemma}\label{monotonicity} 
For $s \geq 0$, 
\[
\mu(B_n, \eta, l + s ,M) \geq \mu(B_n, \eta, l, M).
\]
and 
\[
\mu^*(B_n, \eta, l-s, M) \geq \mu^*(B_n, \eta, l, M). 
\]
\end{lemma} 

\begin{proof}
Let $w \in S(B_n, \eta, l, M)$ . By Lemma \ref{MA-rep},  it suffices to show 
\[
|\partial^+(w, x, B_n)| \leq |\partial^+(w - q_s, x, B_n)|,
\]
for each $x \in B_n$. Choose $p \in \partial^+(w, x)$, if this is not possible, we are done. 
Then,  for each $y \in \bar{B}_n$, 
\begin{align*}
w(x) +  (p - s x ) \cdot(y-x)  + \frac{1}{2} s (|y|^2 - |x|^2)  &= w(x) + p \cdot(y-x) - s x y \\
&+ s |x|^2 + \frac{1}{2} s |y|^2 - \frac{1}{2} s |x|^2 \\
&\geq w(y) + \frac{1}{2} s  |x - y|^2 \\
&\geq w(y).
\end{align*}
And so rearranging, we get 
\[
w(x) - q_s(x) +  (p - s x) \cdot(y-x)  \geq w(y) -q_s(y), 
\]
meaning $p - sx \in \partial^+(w-q_s, x, B_n)$. Since this holds for all $p \in \partial^+(w,x,B_n)$, this implies 
\[
|\partial^+(w, x, B_n)| \leq |\partial^+(w - q_s, x, B_n)|. 
\]
The proof for $\mu^*$ is identical. 
\end{proof}

In the next lemma, we show that if $\mu$ is strictly positive in the limit, then a subsolution must curve downwards in every direction.

\begin{lemma}\label{strict-convexity}

Suppose that $\alpha := \mu(l_M,M) > 0$. There exists a constant $C := C_{\eta_{\min}, \eta_{\max}, l, M, d}$  so that for each $\eta$ in a set $\Omega_{l,M}$ of full probability and $0 < \beta < 1$ the following holds. There is an $n_0 \in \N$ 
so that for all $n \geq n_0$, there exists $w_n \in S(B_{n}, \eta, l, M)$ so that for each $x_0 \in \{ \Gamma_{w_n} = w_n\} \cap B_{\beta n}$ and $p_0 \in \partial^+(w_n, x_0, B_n)$
\[
w_n(y) \leq w_n(x_0) + p_0 \cdot (y-x) - C \alpha n^2
\]
for all $y \in \partial B_{n}$. An analogous result holds for $\mu^*$ with a sign change.

\end{lemma}

\begin{proof}
As $\alpha > 0$, by Lemma \ref{sub-additive-ergodic}, we can choose a set of full probability $\Omega_{l,M}$, so that for every $\eta \in \Omega_{l,M}$ there exists $n_0$ so that for all $n \geq n_0$,  there is $w_n \in S(B_n, \eta, l, M)$ with 
\begin{equation} \label{spreadout}
\frac{\alpha}{2} \leq \frac{ |\partial^+(\Gamma_{w_n},B_{\beta n})|)}{|B_{\beta n}|} \leq \frac{\mu(B_n,l,M)}{|B_{n}|} \leq 2 \alpha
\end{equation}
In light of Lemma \ref{stabilizing-legal}, we can assume 
\[
w_n \in S(B_{n}, \eta, l, M) \cap S^*(B_{n}, \eta, l,M).
\] 
As $|\partial^+(w_n, B_{\beta n})| > 0$, we can find $x_0 \in B_{\beta n}$ with $w_n(x_0) = \Gamma_{w_n}(x_0)$ and $|\partial^+(w_n,x_0)| > 0$. Take $p_0 \in \partial^+(w_n,x_0)$. By a translation and affine transformation, we can suppose $\Gamma_{w_n}(x_0) = 0$, $p_0 = 0$, and $x_0 = 0$. Take $1 > \delta > \beta$. By rescaling and subadditivity, it suffices to show
\begin{equation} \label{strict-concavity}
\Gamma_{w_n}(y) \leq - \alpha C n^2
\end{equation}
for $y \in \partial B_{\delta n}$. Let $\bar{\phi}_n : B_{\delta} \to \R$ be a scaling of the interior of the concave envelope, 
\[
\bar{\phi}_n:= \frac{1}{n^2} \Gamma_{w_n}( [n x]), \mbox{ for $x \in B_{\delta}$}.
\]
As $w_n \in S(B_n, \eta, l, M) \cap S^*(B_n, \eta, l, M)$, we have $|\Delta_{\Z^d} w_n | \leq C$. Hence,  by Proposition \ref{harnack} and the definition of $\Gamma_{w_n}$, $0 \geq \bar{\phi}_n \geq -C$. 

Moreover, as the ball is strictly convex,  $\bar{\phi}_n$ is uniformly Lipschitz in $B_{\delta}$ and hence contains a subsequence which converges uniformly to a concave, continuous function $\bar{\phi}$ (Lemma 1.6.1 in \cite{gutierrez2016monge}). By taking a further subsequence, $\bar{w}_n := \frac{1}{n^2} w_n([n x])$ also converges uniformly to a limit $\bar{w}$. As $\bar{\phi}$ is the concave envelope of $\bar{w}$, it is differentiable 
with Lipschitz gradient and $|D^2 \bar{\phi}| \leq C$ almost everywhere (Lemma 3.3 and Lemma 3.5 in \cite{roberts1995fully}). By Lemma \ref{sub-additive-ergodic} 
and weak convergence of Monge-Amp{\`e}re measures (Lemma 1.6.1 in \cite{gutierrez2016monge}) the subsequential limit, $\bar{\phi}$, must solve a Monge-Amp{\`e}re equation with constant right-hand side $-\alpha$. Hence, $\det D^2 \bar{\phi} = -\alpha$ and in turn $D^2 \bar{\phi} \leq -C \alpha$ almost everywhere. Taking $n_0$ larger if necessary and undoing the scaling,
we have \eqref{strict-concavity}.

\end{proof}

We next use Lemma \ref{strict-convexity} to show Lipschitz continuity of $\mu$.
\begin{lemma}\label{continuity} 
There is a constant $C_{\eta_{\min}, \eta_{\max}, l, M, d}$ so that for all $0 < s < 1$,
\[
\mu(l, M) \leq  \mu(l-s, M)+  s C
\]
and
\[
\mu^*(l,M) \geq \mu^*(l+s,M) + sC. 
\]
\end{lemma}

\begin{proof}
Let $0 < s < 1$ be given. Take $\beta = (1-s)$ and let $C$, $\eta \in \Omega_{l,M}$, and $n \geq n_0$ be given by Lemma \ref{strict-convexity}. Assume $\mu(l,M) > s C$. We will show that after removing a shell of volume proportional to $s$, the set of slopes remaining must be in $\partial^+(w_n + q_s, B_n)$ for all $w_n$ close to achieving the supremum in  $\mu(B_{n}, \eta, l, M)$.

By Lemma \ref{strict-convexity}, there is $w_n \in S(B_{n}, \eta, l, M)$ so that for every $x \in B_{(1-s) n}$ with $\Gamma_{w_n}(x) = w_n(x)$ 
and $p \in \partial^+(w_n, x)$ 
\begin{equation}
w_n(x) + p \cdot(y -x)  \geq w_n(y) + q_{s C}(y), 
\end{equation}
for all $y \in \partial B_{n}$. 
Hence, the argument in the proof of Lemma \ref{ABP} shows that $p \in \partial^+(w_n + q_{s C},B_n)$
and since this applies for all such $p$, 
\[
\partial^+(w_n, B_{(1-s)n}) \subseteq \partial^+(w_n + q_{sC}, B_n) 
\]
Further, using Lemma \ref{mu-is-bounded}, 
\[
|\partial^+(w_n, B_n \backslash B_{(1-s)n})| \leq   s C |B_n|, 
\]
which completes the proof after taking limits.
\end{proof}

The above results show Lipschitz continuity of $\mu$ for each fixed $l \in \R$. Repeating this for every rational $l$ 
in the interval specified by Lemma \ref{mu-is-bounded} and using the intermediate value theorem, we can choose the largest $l_M \in \R$ so that in the limit, 
\[
\mu(l_M, M) = \mu^*(l_M,M),
\]
then define the {\it effective equation} uniquely as
\[
\bar{F}_{\eta}(M) = l_M. 
\]

\subsection{Basic properties of the effective equation}

Here we show that the effective equation is bounded, degenerate elliptic, and Lipschitz continuous. This together with 
the fact any legal stabilizing toppling function has bounded Laplacian will be used in Section \ref{subsec:proof_of_theorem} to establish a comparison principle for solutions to the effective equation. 

\begin{lemma} \label{lemma:basic_effective}
For every $M, N \in \symm^d$, the following hold. 
\begin{enumerate}
\item
Degenerate elliptic:   If $M \leq N$, $\bar{F}_{\eta}(M) \geq \bar{F}_{\eta}(N)$. 
\item Lipschitz continuous:  $|\bar{F}_{\eta}(M) - \bar{F}_{\eta}(N)| \leq C |M-N|_2$. 
\item Bounded:  $|\bar{F}_{\eta}(M)| < \infty$. 
\end{enumerate}

\end{lemma}

\begin{proof}
We show the first inequality. Suppose $N = M + A$ with $A \geq 0$. 
The proof of Lemma \ref{monotonicity}, using $q_A \geq 0$ in place of $q_s \geq 0$, shows that 
$\mu(l_M, M+A) \geq \mu(l_M, M)$ and $\mu^*(l_M, M+A) \leq \mu(l_M, M)$.
By Lemma \ref{monotonicity}, $f(s) := \mu(l_M+s,M+A) - \mu^*(l_M+s,M+A)$, is nondecreasing in $s$
and we have just showed $f(0) \geq 0$.  Hence, $l_{M+A} \leq l_M$ and so $\bar{F}_{\eta}(M+A) \leq \bar{F}_{\eta}(M)$.

For the second inequality, first rewrite, 
\[
\mu(l_M, M) = \mu(l_M, N + (M-N)) = \mu(l_M - |M-N|_2, N + (M-N) + |M-N|_2 I ), 
\]
then observe that $(M-N) + |M-N|_2  I \geq 0$. Hence, by the argument in the first paragraph, 
$\mu(l_M, M) \geq \mu(l_M - |M-N|_2, N)$
and so 
\[
\mu^*(l_M - |M-N|_2, N) \geq \mu^*(l_M,M) = \mu(l_M, M) \geq \mu(l_M - |M-N|_2, N).  
\]
and hence 
\[
\bar{F}_{\eta}(N) \geq \bar{F}_{\eta}(M) - |M-N|_2.
\]
Swapping the roles of $M$ and $N$ then show  
\[
|\bar{F}_{\eta}(M) - \bar{F}_{\eta}(N)|  \leq |M-N|_2. 
\]
The third claim follows by construction and Lemma \ref{mu-is-bounded}.

\end{proof}

\section{Proof of the Theorem}\label{sec:proof_of_theorem}
For each $n \in \N$, recall that
\[
v_n = \min \{ v : \Z^d \to \N: \Delta_{\Z^d} v_n + \eta I(\cdot \in W_n) \leq 2 d-1 \}, 
\]
is the odometer function for $\eta$ on $W_n$ with the {\it free} boundary condition and $\bar{v}_n = n^{-2} v_n([n x])$ is its rescaled linear interpolation.  
We start by showing that $\bar{v}_n$ is equicontinuous and bounded. Then, we show that the high density assumption, $\E(\eta(0))  > 2 d -1$, 
implies $v_n \geq 1$ in $W_{n - o(n)}$, enabling an essential tool in the proof of Lemma \ref{flat}, (Dhar's burning algorithm, Lemma \ref{burning}). 
We then conclude by showing that every scaled subsequence converges to the same limit. 

\subsection{An upper bound on the odometer function} \label{subsec:compactness}
We establish an upper bound on $\bar{v}_n$ by constructing a toppling function 
which stabilizes $\eta_{\max}$ and hence $\eta$. Since $\eta_{\max}$ is constant, we can stabilize by toppling 
\lq one dimension at a time\rq, a trick from \cite{ fey2010growth}, and restated 
below for the reader. (Note one could also compare to the divisible sandpile as in \cite{levine2009strong} to get a tighter bound).
\begin{lemma}[Lemma 3.3 in \cite{ fey2010growth}] \label{lemma:one_dim}
	Let $\ell \in \N$ be given. Pick $k \in \N$ so that $R_k := \eta_{\max} - (2d-k) = 2 r $ for some $r \in \N$. Then, 
	there exists $g: \Z \to \N$ so that 
	\[
	\Delta_{\Z^d} g = f, 
	\]
	where  $f: \Z \to \Z$ is given by 
	\[
	f(x) = 
	\begin{cases} 
	-R_k &\mbox{ for $|x| \leq \ell$ }  \\
	2 &\mbox{ for $\ell < |x| \leq \ell(r + 1) $} \\
	1 &\mbox{ for $\ell(r+1)  <  |x| \leq \ell(r+1) + r$} \\
	0 &\mbox{ for $ \ell(r+1) + r< |x|$}
	\end{cases}
	\]
	Moreover, $g$ is supported in $I = \{ x \in \Z: |x| < \ell(r+1) + r\}$ and there exists $C:= C_{r}$ for which 
	\begin{equation} \label{quadratic bound}
	g(x) \leq C x^2. 
	\end{equation}
\end{lemma}
We use this technique together with the universal bound on the Laplacian to show compactness of $\bar{v}_n$. 
\begin{lemma}\label{compactness}
For every subsequence $n_k \to \infty$ there is a subsequence $n_{k_j}$
and a function $\bar{v} \in C(\R^d)$ so that $\bar{v}_{n_{k_j}} \to \bar{v}$
uniformly as $j \to \infty$. 
\end{lemma}

\begin{proof}

Cover $W_n$ with a box of side length $C_{d, W} n$ for some $C_{d, W} \in \N$. Choose $g$ from Lemma \ref{lemma:one_dim} with $\ell = C_{d, W} n$ and 
for $x = (x_1, \ldots, x_d) \in \Z^d$, define 
\[
u_i(x) = g(x_i), 
\]
and observe that by definition of $g$, $\Delta_{\Z^d} u_i+ \eta_{\max} \leq 2 d -1$.  Hence, by the least action principle, as $\min(u_1, \ldots, u_d)$ is also stabilizing, 
\[
v_n(x) \leq \min(u_1(x), \ldots, u_d(x)) \leq C_d |x|^2.
\]
Hence, $\bar{v}_n \leq C_d$ and is supported in $Q_{C_{d,W}}$.  We have equicontinuity since $|\Delta_{\Z^d} v_n| \leq C_{d, \eta_{\min}, \eta_{\max}}$ (\cite{kuo2005estimates}). 
The Arzela-Ascoli theorem now implies the claim. 
\end{proof}

\subsection{A lower bound on the odometer function}
In this subsection, we use a comparison principle argument to show that on an event of probability 1,  $v_n \geq 1$ in $W_{n - o(n)}$.
As a corollary, this argument gives a quantitative proof of the (now classical) fact that if $\E(\eta(0)) > 2 d -1$ then $\eta$ is almost surely {\it exploding}, (see \cite{fey2009stabilizability}). 
The technique takes inspiration from Theorem 4.1 in \cite{levine2009strong}. In essence, the proof is a comparison of $v_n$ with the odometer function for the random divisible sandpile with threshold $2 d-1$. See Section \ref{sec:divisible_sandpile} for more on the random divisible sandpile, including a proof of convergence which uses 
Lemma \ref{divisible_homogenize}. 

We start by briefly recalling the Green's function for simple random walk on $\Z^d$ stopped when exiting the ball .
Let $S_n^{(x)}$ be simple random walk started at a site $x$ in $\Z^d$ and let $\tau_n = \min\{t \geq 0: S_{t} \not \in B_n \}$.
Let 	
\[
g_n(x,y) = \frac{1}{2d } \E \sum_{n=0}^{\tau_{n}-1} 1\{S_n^{(x)} = y\}, 
\]
Fix $\delta > 0$ and $x \in B_n$. From \cite{lawler2010random}, we have the following exit time estimates, 
\begin{equation} \label{exit_time}
\begin{aligned}
\sum_{y \in B_n} g_n(x,y)  &= O(n^2) \\ 
\sum_{z \in B_{\delta n}} g_n(x,x+z) &= \delta O(n^2)
\end{aligned}
\end{equation}
and the following difference estimates, for $\max(|x|,|y|) < (1-\delta) n^2$,
\begin{equation} \label{difference}
|g_n(x,y) - g_n(x,y + e_i)| = O(|x-y|^{1-d}) + O_{\delta}(n^{2-2d}).
\end{equation}
Next, define for each $n$ 
\[
r_n(x) := \sum_{y \in B_{n}} g_{n}(x,y) \eta(y),
\]
\[
d_n(x) :=  \sum_{y \in B_{n}} g_{n}(x,y) \E(\eta(0)).
\]
so that $\Delta r_n(x) = - \eta(x)$, $\Delta d_n(x) = -\E(\eta(0))$, and $r_n(y) = d_n(y) = 0$ on $\partial B_n$. 
The next lemma uses these estimates together with the ergodic theorem to show that $r_n$ and $d_n$ are identical in the scaling limit. 

\begin{lemma}\label{divisible_homogenize}
For each $\eta \in \tilde{\Omega}_0$, an event of probability 1, there is a constant $C := C_{d, \eta}$ so that the following holds. 
For each $\epsilon > 0$, there exists $n_0 \in \N$ so that for all  $n \geq n_0$, 
\begin{equation}\label{eq:divisible_homogenize}
\sup_{x \in B_n} \left| r_n(x) -d_n(x) \right| \leq \epsilon C n^2
\end{equation}
\end{lemma}


\begin{proof} 
	 Let $1 > \epsilon > 0$ be given.  Fix dyadic rational $\epsilon > \beta > 0$ small. By Proposition \ref{sub-additive-ergodic} there is an event of full probability, $\tilde{\Omega}_0$, so that for each $\eta \in \tilde{\Omega}_0$, for all $n \geq n_0$, 
\begin{equation}\label{ergodic}
\left( \E(\eta(0)) - \epsilon \right) \leq \frac{1}{|A_{\beta n}(z_i)|} \sum_{y \in A_{\beta n}(z)} \eta(y) \leq \left(  \E(\eta(0)) + \epsilon \right), 
\end{equation} 
 for all $A_{\beta n}(z_i) \subset B_{n}$ which are defined in the following way. 
Take a dyadic partition of disjoint cubes of radius $\beta$ which cover
$Q_1$ in $\R^d$, remove cubes which do not overlap $B_1$, and delete parts of the cubes which are outside $B_1$. Label each cube by an interior point $z_i \in B_n$ and enumerate them as $\{A_{\beta}(z_i)\}$. For each $z_i$, let its finite difference approximation be $A_{\beta n}(z_i) = n A_{\beta}(z_i) \cap B_n$ (delete overlapping boundaries if needed).

Rewrite, 
\begin{equation}\label{eq1}
r_n(x) - d_n(x) = \sum_{A_{\beta n}(z_i) \subset B_{n}} \sum_{y \in A_{\beta n}(z_i)} g_{n}(x,y) (\eta(y) - \E(\eta(0))).
\end{equation}
The rest of the argument is roughly the following. Imagine a non-random sandpile, $\eta_{avg}$,
in which $\eta_{avg} := \E(\eta(0))$ {\it divisible} grains are at each coordinate in $B_n$. In each subcube, $A_{\beta n}(z_i) \subset B_n$, we try to rearrange the grains in the random sandpile, $\eta$, to 
match the deterministic sandpile $\eta_{avg}$. It's possible that there aren't enough grains to do this, so we add just enough for it to match $\eta_{avg}$.  By \eqref{ergodic}, we need to add at most $\epsilon |A_{\beta n}(z_i)|$ grains to each subcube.
Hence, by the exit time estimate, the total cost associated with adding grains is of order $\epsilon O(n^2)$, by the difference estimate, the total cost of rearranging grains within each subcube is of order $o(n^2)$, leading to \eqref{eq:divisible_homogenize}.

 Here are the details. If $x \in B_{n} \backslash B_{(1-\beta) n}$, 
 by comparing to a quadratic, $\max(r_n, d_n)(x) \leq C \beta n^2$. Hence, 
 we may suppose $x \in B_{(1-\beta) n}$. First, we add $\epsilon$ grains to every site in the cube, this incurs an error which we can bound using \eqref{exit_time},
  \begin{equation}
 \mathcal{E}_1(x) = \sum_{y \in B_n} \epsilon g_n(x,y) \leq \epsilon C_d n^2.
 \end{equation}
 Then, we start rearranging. First, we remove a constant number of cubes near $x$,
	\begin{equation}
	\mathcal{A}_x = \{ A_{\beta n}(z_i) : \inf_{y \in A_{\beta n}(z_i)} |y - x| < \beta n \}, 
	\end{equation}
	by adding $\eta_{avg}-\eta_{\min}$ grains to each site in the subcubes, 
	\begin{equation}
	\mathcal{E}_2(x) = \sum_{A_{\beta n}(z_i) \in \mathcal{A}_x} \sum_{y \in A_{\beta n}(z_i)}  (\eta_{avg}-\eta_{\min}) g_n(x,y) \leq  C \beta n^2, 
	\end{equation} 
	using \eqref{exit_time}.

 Now, we rearrange the random assortment of grains in all other subcubes $A_{\beta n}(z_i)$ so that the number of grains 
 at every site is $\eta_{avg}$. We start by pooling every grain at sites $y \in A_{\beta n}(z_i)$ 
 to $z_i$, this incurs an error of 
 \begin{equation}
 \mathcal{E}_3(x,y) = (\eta(x)+\epsilon) ( g_n(x,y) - g_n(x,z_i)).
 \end{equation}
 Then, we move $\eta_{avg}$ grains from $z_i$ back to $y$, with error
  \begin{equation}
 \mathcal{E}_4(x,y) = \eta_{avg} (g_n(x,z_i) - g_n(x,y)). 
 \end{equation}
 We can iterate \eqref{difference} to see that  
 \begin{equation} 
 \sup_{x,y} |\max(\mathcal{E}_3(x,y), \mathcal{E}_4(x,y))| \leq C_d \beta n \sup_{z \in A_{\beta n}(z_i)} |x-z|^{1-d}, 
 \end{equation} 
 And, by an integral approximation, 
 \begin{equation}
 C_d \beta n \sum_{z_i \not \in \mathcal{A}_x} \sup_{z \in A_{\beta n}(z_i) } |x-z|^{1-d} \leq C_d \beta^{-1} n = o(n^2).  
 \end{equation}
 For each $y \in B_n$, write,
 \begin{align*}
 g_n(x,y) \eta(y) &=  -g_n(x,y) \epsilon \\
 &+ (\eta(y) + \epsilon)(g_n(x,y) - g_n(x,z_i)) \\
 &+ (\eta(y) + \epsilon)  g_n(x,z_i) \\ 
 &+ \eta_{avg} (g_n(x,z_i) - g_n(x,y)) \\
 &- \eta_{avg} (g_n(x,z_i) - g_n(x,y)).  
 \end{align*}
 Putting this together, 
  \begin{align*}
   &\sum_{A_{\beta n}(z_i) \subset B_{n}} \sum_{y \in A_{\beta n}(z_i)} g_{n}(x,y) \eta(y) \\
   &\geq \sum_{A_{\beta n}(z_i) \subset B_{n}} \sum_{y \in A_{\beta n}(z_i)} g_{n}(x,y) \eta_{avg} \\
   &+ (-\epsilon C n^2)  \\
   &+ \sum_{ \{ A_{\beta n}(z_i) \subset B_{n}\} \backslash \mathcal{A}_x} g_n(x,z_i) \left( \sum_{y \in A_{\beta n}(z_i)}  (\eta(y) + \epsilon - \eta_{avg}) \right)  \\
   &\geq -\epsilon C n^2 + \sum_{A_{\beta n}(z_i) \subset B_{n}} \sum_{y \in A_{\beta n}(z_i)} g_{n}(x,y) \eta_{avg}
   \end{align*}
   Where we used the fact $\inf_{ \{x,y\} \in B_n } g_n(x,y) \geq 0$. The other direction follows by swapping the roles of $\eta$ and $\eta_{avg}$.

\end{proof}

We next use this to provide the desired lower bound on $v_n$. 

\begin{lemma}\label{lowerbound}

For each $\eta \in \tilde{\Omega}_0$, an event of probability 1, and each $\epsilon > 0$, there exists $n_0$ so that for all $n \geq n_0$, 
\[
w_n \geq 1 \mbox{ for all $x \in B_{(1-\epsilon)n}$}, 
\]
where
\[
w_n \in \mathcal{L}(\eta, B_{n}) \cap \mathcal{S}(\eta,B_{n}) \mbox{ and $w_n = 0$ on $\partial B_{n}$}. 
\]

\end{lemma}


\begin{proof}
Let $\epsilon > 0$ be given
and   
\[
\delta := \left( \E(\eta(0)) - 2 d - 1\right) > 0. 
\]
Choose $\eta \in \tilde{\Omega}_0$, $C$, and $n \geq n_0$ from Lemma \ref{divisible_homogenize} with $\epsilon'  > 0$ small
to be chosen below. As $w_n - r_n - q_{2 d- 1}$ is superharmonic in $B_{n}$,  for $x \in B_{n}$, 
\begin{align*}
w_n(x) - r_n(x) - q_{2 d- 1}(x)  &\geq  \min_{ y \in \partial B_{n}}  \left( w_n(y) -r_n(y)- q_{2 d- 1}(y) \right) \\
&= \min_{ y \in \partial B_{n}} -q_{2 d- 1}(y),
\end{align*}
using  $w_n = r_n = 0$ on $\partial B_{n}$.
Hence, Lemma \ref{divisible_homogenize} then shows 
\[
w_n(x) \geq d_n(x)  -(2d-1) \frac{1}{2}(n^2-|x|^2) - \epsilon' C n^2 \\
\]
By assumption, $d_n + q_{2 d- 1+\delta}$ is superharmonic in $B_{n}$ and so 
\[
d_n(x) + q_{2 d- 1+\delta}(x) \geq  \min_{ y \in \partial B_{n}}\left( d_n(y) + q_{2 d- 1+\delta}(y) \right).
\]
Using again $d_n(y) =0$ on $\partial B_n$, 
\[
w_n(x) \geq  \delta/2 (n^2 - |x|^2) - \epsilon' C n^2.  
\]
In particular, we can choose $\epsilon'$ small and $n_0$ large so that
\[
w_n(x) \geq 1
\]
for $x \in B_{(1-\epsilon) n}$ . 
\end{proof}


\subsection{A comparison principle in the limit} 
In order to compare subsequential limits of the odometer for different $\eta$ we must show that $\mu(l_M, M) = \mu^*(l_M,M) = 0$.
The argument is roughly this: if both $\mu$ and $\mu^*$ are strictly positive in the limit, then there is a subsolution and supersolution whose difference bends upwards in every direction.
However, when there are enough topples, this difference obeys a comparison principle on the microscopic scale, due to Proposition \ref{burning},
and so this cannot happen.

\begin{lemma} \label{flat}
$\mu(l_M,M) = \mu^*(l_M,M) = 0$
\end{lemma}

\begin{proof}
	By definition, $\mu(l_M,M) = \mu^*(l_M,M) \geq 0$. We will show that it is impossible for both $\mu(l_M,M)$ and $\mu^*(l_M,M)$ to be strictly positive. Suppose for sake of contradiction that $\mu(l_M, M) = \mu^*(l_M,M) = \alpha > 0$.

	As $\alpha > 0$ we can invoke Lemma \ref{strict-convexity} for both $\mu$ and $\mu^*$. Take $0 < \beta < 1$ small. Using the fact $\mu$ converges evenly over the unit ball, (see Lemma 3.2 in \cite{armstrong2014stochastic}), 
	we may select $v, u \in \mathcal{L}(B_n, \eta) \cap \mathcal{S}(B_n, \eta)$ 
	with $|\partial^-(v,B_{\beta n},B_n)|>0$ and $|\partial^+(u,B_{\beta n}, B_n)| > 0$
	for which the claims in Lemma \ref{strict-convexity} apply. (Note we used Lemma \ref{stabilizing-legal} to pick locally legal {\it and} stabilizing toppling functions.)

	 Moreoever,  as $\mu$ and $\mu^*$ are invariant under affine transformations,  we can then choose affine functions $L_u$ and $L_v$ so that
	\begin{align} \label{positive}
	&\inf_{x \in B_{n}}  -(u - q_{M} + L_u)(x) = (u - q_{M} + L_u)(x_0) = 0 \\
	&\inf_{x \in B_{n}}  (v - q_{M} + L_v)(x) = (v - q_{M} + L_v)(x_0^*) = 0 \nonumber,
	\end{align}
	for some $x_0, x_0^* \in B_{\beta n}$ and 
	\begin{align} \label{curveup}
	& -(u - q_{M} + L_u)  \geq C n^2  \mbox{ on } \partial B_{n}  \\ 
	& (v - q_{M} + L_v)  \geq C n^2  \mbox{ on } \partial B_{n} \nonumber.
	\end{align} 

	Now, use the Abelian property, Proposition \ref{Abelian}, to decompose $u$ and $v$ into the initial toppling of $\eta$ and then topplings originating from the boundary, 
	$u = u_1 + w$ and $v = v_1 + w$. By Lemma \ref{lowerbound} and Proposition \ref{recurrent}, (moving the boundary of the ball inwards if necessary and accumulating an $o(n^2)$ error),
	$\Delta_{\Z^d} w + \eta$ is recurrent in $B_n$. Now, approximate $L_v(x) = p \cdot x +r$  by 
	\[
	\tilde{L}_v(x) = [p] \cdot x + [r],
	\]
	an integer-valued function, (this approximation also incurs an $o(n^2)$ error). Repeat for $L_u$ with $\tilde{L}_u$.  Hence, by Proposition \ref{burning} and \eqref{curveup}
	\begin{align*}
	&\left( (v + \tilde{L}_v - q_M) - (u + \tilde{L}_u - q_M) \right)(0)  \\
	&=   \left( (v_1 + \tilde{L}_v )- (u_1 + \tilde{L}_u) \right)(0)  \\
	&\geq \inf_{y \in \partial B_{n}}   \left( (v_1 + \tilde{L}_v )- (u_1 + \tilde{L}_u) \right)(y) \\
	&=  \inf_{y \in \partial B_{n}}   \left( (v + L_v - q_M)- (u + L_u - q_M) \right)(y)  - o(n^2) \\
	&\geq C n^2.
	\end{align*}
	However, this contradicts the Harnack inequality for $n$ large and $\beta$ small.  Indeed, due to \eqref{positive} and 
	\[
	\max(|\Delta_{\Z^d} (v - q_{M} + L_v)|, |\Delta_{\Z^d}  (u - q_{M} + L_u))| \leq C,
	\]
	we can apply the Harnack inequality, Lemma \ref{harnack}, to see
	\begin{equation} \label{small}
	\left( (v + L_v - q_M) - (u + L_u - q_M) \right)(0)  \leq C \beta n^2 
	\end{equation}
	as $x_0, x_0^* \in B_{\beta n}$.

\end{proof}

\subsection{Proof of Theorem \ref{thetheorem}} \label{subsec:proof_of_theorem}
Choose $\Omega_0$  to be the intersection of $\Omega_{l,M}$ in Lemma \ref{sub-additive-ergodic} 
over all $l \in \R$ and $M \in \symm^d$ with rational entries and $\tilde{\Omega}_0$ from Lemma \ref{lowerbound}. 
Pick $\eta, \eta' \in \Omega_0$ and choose subsequences of scaled odometers $\bar{v}_n$ and $\bar{v}_n'$ corresponding to $\eta$ and $\eta'$
with {\it free boundaries} which converge uniformly to $v$ and $v'$. Suppose for sake of contradiction that $v \not = v'$. Since $v = v' = 0$ outside $B_R$ for some $R > 0$, we may assume without loss of generality that 
\[
\sup_{B_R} (v - v') >  0 = \sup_{\partial B_R} (v - v')
\]
We restate for the reader results contained in \cite{pegden2013convergence}. 
\begin{lemma} \cite{pegden2013convergence}  \label{regularity}
\begin{enumerate}
\item  There exists $a \in \R^d$ either in $W$ or outside the closure of $W$ so that $v(a) > v'(a)$, both $v$ and $v'$ are twice differentiable at $a$ and
$D^2(v - v')(a) < -\delta I$ for some $\delta > 0$.  
\item  For each $\epsilon > 0$, if $a$ is outside the closure of $W$,  we may select $u : \Z^d \to \Z$ such that 
\[
\Delta_{\Z^d} u(x) \leq 2 d -1 \mbox{ and } u(x) \geq \frac{1}{2} x^T (D^2 v(a) - \epsilon  I) x \mbox { for all $x \in \Z^d$ }. 
\]
\item For each $\epsilon > 0$, if $a$ is in $W$, we may select $u: \Z^d \to \Z$ such that
\[
\Delta_{\Z^d} u(x) \leq 2 d -1 \mbox{ and } u(x) \geq \frac{1}{2} x^T (D^2 v(a) - \epsilon  I) x + o(|x|^2) \mbox { for all $x \in \Z^d$ }. 
\]
\end{enumerate}
\end{lemma}

\begin{proof}
The first and second statements are Proposition 2.5 and Lemma 4.1 in \cite{pegden2013convergence}. We sketch the third. 
  For each $\epsilon > 0$, the proof of Lemma 4.1 in \cite{pegden2013convergence} gives a function 
\[
u : \Z^d \to \Z 
\]
with 
\[
u(x) \geq \frac{1}{2} x^T (D^2 v(x_0) - \epsilon I)  x. 
\]
and 
\[
\Delta_{\Z^d} u + \tilde{\eta} \leq 2 d -1  
\]
where $\tilde{\eta}$ is a periodic tiling of $\eta$ in $B_{ r n }$ for some $r > 0$ and $n \in \N$ large. 
Due to Lemma \ref{sub-additive-ergodic}, picking $n$ larger if necessary, we have 
\[
\frac{1}{B_{r n}} \sum_{x \in B_{r n}} \eta(x) \geq 2 d -1
\]
Hence, by Rossin's observation \cite{rossin2000proprietes} (see Fact 3.5 in \cite{levine2016apollonian}), as a sandpile configuration on $\Z^d$, $\Delta_{\Z^d} u$ is stabilizable, and so by toppling it, we find a subquadratic, finite $w: \Z^d \to \N$ so that 
\[
\Delta_{\Z^d} (u + w)  \leq 2 d -1, 
\]
and $(u + w)(x) = q_{D^2 v(a) - \epsilon}(x)+ o(|x|^2)$. 

\end{proof}

Now, let $a \in \R^d \backslash \partial W$ be a given point satisfying the properties in part 1 of Lemma \ref{regularity}. 
If $a$ is outside the closure of $W$, the argument in the proof of Theorem 4.2 in \cite{pegden2013convergence} which uses part 2 of Lemma \ref{regularity} leads to a contradiction. So, it suffices to suppose $a \in W$. 
In this case, we cannot use the same argument to compare $v$ and $v'$ as they stabilize (possibly) different random sandpiles. Instead, we use $\mu$ to compare the two.

Since $v'$ is twice differentiable at $a$, by Taylor's theorem, 
\[
v'(x) = \phi(x) + o(|x-a|^2)
\]
where 
\[
q_M + L_{\phi} := \phi(x) := v'(a) + D v'(a) \cdot (x - a) +  \frac{1}{2} (x - a)^T D^2 v'(a)(x - a)
\]
Pick the unique $l := \bar{F}_{\eta}(D^2 v'(a)) \in \R$ so that 
\[
\mu(l, D^2 v'(a)) = \mu^*(l, D^2 v'(a)) = 0.
\]
By approximation, (using Lemma \ref{lemma:basic_effective}), we can assume $M$ and $l$ are rational. Then, by Lemma \ref{ABP}, (recalling that $\mu$ is invariant under affine transformations), for all small $r > 0$
and $n \in \N$ large,
\begin{align*}
\inf_{x \in B_{r n}(a)}  \left( v_n  - q_M - n L_\phi -q_l \right)(x) &\geq \inf_{y \in \partial B_{r n}(a)} \left( v_n  - q_M - n L_\phi - q_l \right)(y) \\
&- C_d n \mu^*(B_{r n}, \eta, 0, M)^{1/d}.
\end{align*}
And so, after rescaling, 
\begin{align*}
\inf_{x \in n^{-1} B_{r n}(a)} \left( \bar{v}_n  - \phi - q_l \right)(x) &\geq \inf_{y \in \partial n^{-1} B_{r n}(a)} \left( \bar{v}_n  - \phi - q_l \right)(y) \\
&- \left( \frac{C_d n \mu^*(B_{r n}, \eta, 0, M)^{1/d} }{n^2}\right)
\end{align*}
which implies by uniform convergence of $\bar{v}_n \to v$ and Lemma \ref{sub-additive-ergodic},
\[
\inf_{x \in B_r^{}(a)} \left( v- \phi - q_l \right)(x) \geq \inf_{y \in \partial B_{r} (a) } \left( v - \phi - q_l \right)(y). 
\]
In particular, 
\begin{align} \label{almost_strict_max}
(v - v' - q_l)(a)  &= (v - \phi - q_l)(a) \\ \nonumber
&\geq \inf_{y \in \partial B_{r} (a) } \left( v - \phi - q_l \right)(y)  \\ \nonumber
&=  \inf_{y \in \partial B_{r} (a) } \left( v - v' - q_l \right)(y)  - o(r^2) \nonumber
\end{align}
If $l \leq 0$, sending $r \to 0$ in \eqref{almost_strict_max} contradicts $D^2(v-v')(a) < -\delta I$. Hence $l > 0$. However, the same argument, applying Lemma \ref{sub-additive-ergodic} to $\mu$ and $v'$ shows, 
\[
(v' - \phi - q_l)(a) \geq  \inf_{y \in \partial B_{r}(a)} (v' -\phi-q_l)(y) 
\]
which contradicts Taylor's theorem for $r$ small as $v'$ and $q_l$ are twice differentiable at $a$.





\section{Convergence of the random divisible sandpile} \label{sec:divisible_sandpile}
One of the challenges involved in the Abelian sandpile model is the integrality constraint on the odometer function. 
In the {\it divisible sandpile} model, this constraint is relaxed and sites are allowed to topple a fractional number of times. 
This relaxation enables the use of simple random walk estimates which leads to a more direct proof of convergence. 

\subsection{Description of the divisible sandpile}
We briefly describe the divisible sandpile, referring the interested reader to \cite{levine2010scaling,levine2016divisible} for more details. 
Begin with some, possibly fractional, distribution of sand and holes, on a domain $V \in \Z^d$,  $\eta : V \to \R$.  A site $x \in V$ is unstable whenever $\eta(x) > 1$, 
in which case the {\it excess mass}, $1 - \eta(x)$, is equally distributed among the neighbors of $x$ until every site is stable. The odometer function, 
$v$, then counts the total mass emitted by each site. Here, the starting point is also a discrete PDE: the 
{\it least action principle for the divisible sandpile}. 
\begin{prop}[Proposition 2.5 in \cite{levine2016divisible}] \label{prop:div_lap}
\[
v = \min\{ f: \bar{V} \to \R^+: \Delta_{\Z^d} f + \eta \leq 1 \}
\]
\end{prop}

\subsection{Convergence of the odometer function}

As in Section 2, we consider a stationary, ergodic, probability space $(\Omega, \mathcal{F}, \mathbf{P})$, with $\Omega$ the set of all bounded backgrounds, 
\[
\eta : \Z^d \to \R
\]
for which 
\[
\sup_{x \in \Z^d} \eta(x) < \infty. 
\]
In this case, we do not require $\eta$ to be high density, but we do assume for simplicity uniform boundedness: there exists $\eta_{\min}, \eta_{\max} \in \R$ so that for every $x \in \Z^d$, 
\begin{equation}
 \mathbf{P}\left[ \eta_{\min} \leq \eta(x) \leq \eta_{\max} \right] = 1. 
\end{equation} 
Let $W \subset \R^d$ be a bounded Lipschitz domain. For each $n \in \N$, let $W_n = \Z^d \cap n W$ denote the discrete approximation of $W$. 
Initialize the sandpile according to $\eta(x)$ in $W_n$ and let $v_{n}$ be its odometer function (defined via Proposition \ref{prop:div_lap}).  Next, consider the averaged initial sandpile, 
\[
\eta_{avg}  := \E \eta(0), 
\]
and the corresponding odometer function, $v_{avg_n}$ for $\eta_{avg}$ in $W_n$.   For the reader's convenience, we restate the form of 
Lemma \ref{divisible_homogenize} we use. Let $g_n(x,y)$ be the Green's function for simple random walk started at $x$ stopped when exiting $W_n$, and 
\[
r_n(x) := \sum_{y \in W_{n}} g_{n}(x,y) \eta(y),
\]
\[
d_n(x) :=  \sum_{y \in W_{n}} g_{n}(x,y) \eta_{avg}.
\]
\begin{lemma}\label{divisible_homogenize2}
There exists a constant $C := C_d$ so that on an event of full probability,  
for each $\epsilon > 0$, there exist $n_0 \in \N$ so that for all  $n \geq n_0$, 
\begin{equation}\label{eq2}
\sup_{x \in \bar{W}_{n}} \left| r_n(x) -d_n(x) \right| \leq \epsilon C_d n^2
\end{equation}
\end{lemma}

Levine and Peres showed in \cite{levine2010scaling} that $\bar{v}_{avg_n}$ converges uniformly to the solution of a linear PDE. 
So, in order to show that $\bar{v}_n$ has a scaling limit, it suffices to show that it stays close to $\bar{v}_{avg_n}$ for all large $n$. Most of the work for this proof is done in Lemma \ref{divisible_homogenize2}, all that's left is a use of the least action principle for the divisible sandpile.
\begin{theorem} \label{divisible_theorem}
On an event of full probability, as $n \to \infty$, the rescaled functions $\bar{v}_n:= n^{-2} v_n([n x])$  and $\bar{v}_{avg_n}:= n^{-2} v_{avg_n}([n x])$ converge uniformly together, 
\[
\sup_{x \in n^{-1} \bar{W}_n} |\bar{v}_n(x) - \bar{v}_{avg_n}(x)| \to 0.
\]
\end{theorem}

\begin{proof}
By definition, 
\[
\Delta_{\Z^d} v_n + \eta \leq 1, 
\]
in $W_n$, which can be rewritten as 
\[
\Delta_{\Z^d} (v_n - (r_n - d_n)) + \eta_{avg} \leq 1. 
\]
Let $\epsilon > 0$ be given. For $n$ large, 
Lemma \ref{divisible_homogenize2} implies $-(r_n - d_n) + \epsilon C n^2$ is positive in $W_n$. 
Hence, by the least action principle in $W_n$, 
\[
v_n - (r_n - d_n) + \epsilon C n^2  \geq v_{avg_n}
\]
and so, 
\[
v_n - v_{avg_n} \geq (r_n - d_n) - \epsilon C n^2. 
\]
Scale and invoke Lemma \ref{divisible_homogenize2} again to see that
\[
\bar{v}_n - \bar{v}_{avg_n} \geq -\epsilon C. 
\]
The other direction is identical. 
\end{proof}

\section{Concluding remarks} \label{sec:conclusion}
We conclude with some straightforward extensions of our results and open questions.

\subsection{Sandpiles with open boundaries}
The exact same argument given in this paper also works for sandpiles with the open boundary condition.
\begin{theorem}
Let $W$ be a bounded Lipschitz domain and let $v_n$ be the odometer function for the 
sandpile $W_n := \Z^d \cap n W$ with the open boundary condition: 
\[
v_n  \in \mathcal{L}(\eta, W_n) \cap \mathcal{S}(\eta, W_n) \mbox{ and } v_n = 0 \mbox{ on $\partial W_n$}. 
\] 
Almost surely, as $n \to \infty$, the rescaled functions $\bar{v}_n:= n^{-2} v_n([n x])$ converge uniformly to the unique viscosity solution $\bar{v} \in C(\R^d)$ of the deterministic equation
\[
\begin{cases}
\bar{F}_{\eta}(D^2 \bar{v} ) = 0 &\mbox{ in $W$} \\
\bar{v} = 0  &\mbox{ on $\partial W$},
\end{cases}
\]
where $\bar{F}_\eta$ is a unique degenerate elliptic operator. 
\end{theorem}

\begin{figure}
\begin{centering}
 \includegraphics[width=0.32\textwidth]{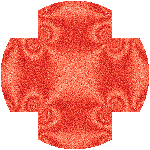} \\
 \includegraphics[width=0.32\textwidth]{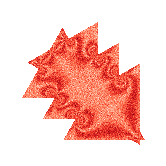}
  \includegraphics[width=0.32\textwidth]{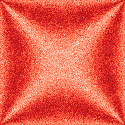}
 \caption{Start with an iid Bernoulli(3,5,1/2) sandpile configuration and stabilize with the open boundary condition. Darker reds are closer to 2 while lighter reds are closer to 3. The displays are approximations of the weak-* limits.}
  \label{closed_boundary}
\end{centering}
\end{figure}

Note that $\bar{F}_\eta$ is the {\it same} operator appearing in the limit of the free boundary sandpile. For example, if the background is the product Bernoulli measure, simulations reveal interesting pictures. These may help characterize $\bar{F}_\eta$ - see Figure \ref{closed_boundary}.

\subsection{Single source sandpile on a random background}
\begin{figure}
\begin{centering}
  \includegraphics[width=0.5\textwidth]{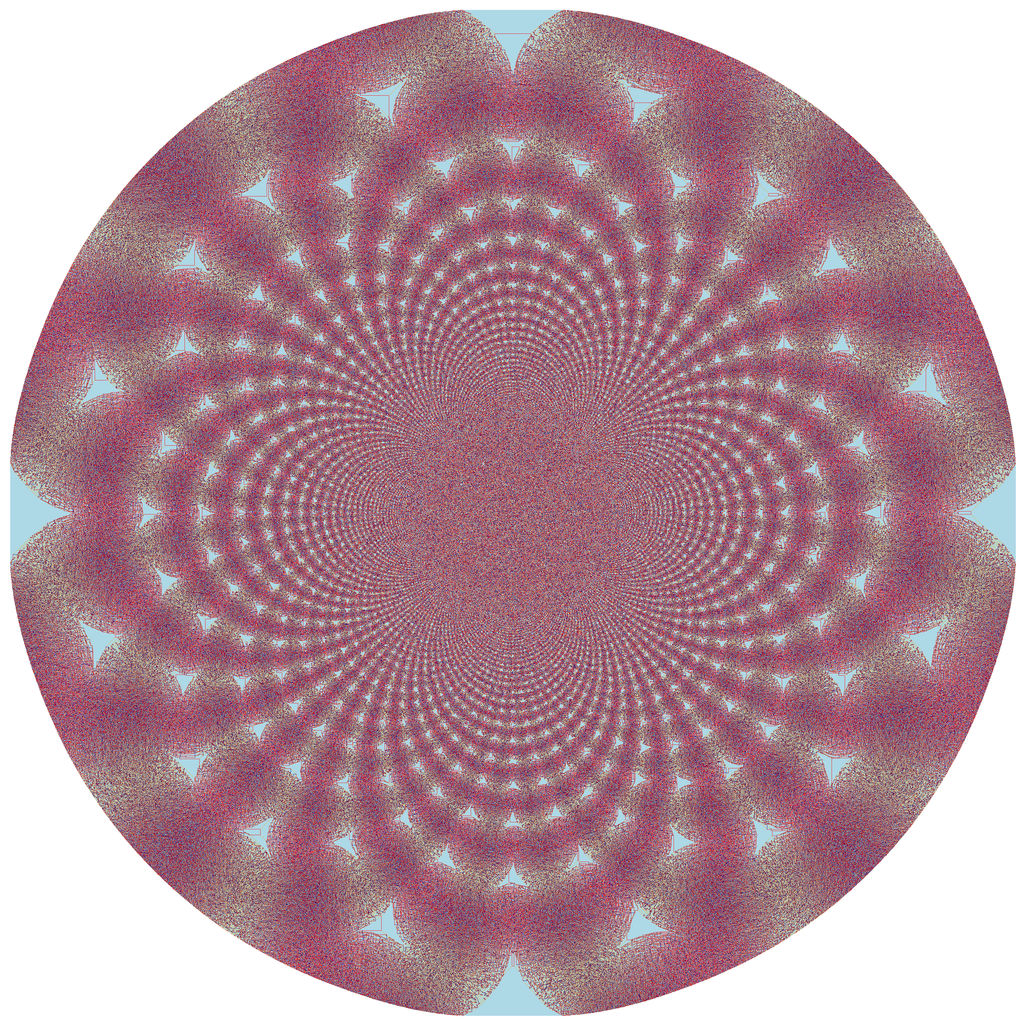}
 \caption{Start with $2^{25}$ chips at the origin in $\Z^2$ with an iid Bernoulli(0,-1,1/2) background and stabilize. What's displayed is an approximation of the weak-* limit.}
  \label{fig_SS_rand}
\end{centering} 
\end{figure}

Straightforward modifications of the arguments appearing above and in \cite{pegden2013convergence} show that single source sandpiles on random backgrounds also have scaling limits. 
See Figure \ref{fig_SS_rand} for an example. 

\begin{theorem}
Let $v_n$ be the odometer function for the sandpile with $n$ chips at the origin 
on a stationary, ergodic, random background $\eta_{\min} \leq \eta \leq \eta_{\max} = 2d-2$, 
\[
v_n  \in \mathcal{L}(\eta + n \delta_0 , \Z^d) \cap \mathcal{S}(\eta + n \delta_0, \Z^d) .
\] 
Almost surely, as $n \to \infty$, the rescaled functions $\bar{v}_n:= n^{-2/d} v_n([n^{1/d} x])$ converge locally uniformly away from the origin to $\bar{v} + G$, where $G$ is the fundamental solution of the Laplacian in $\R^d$ and $\bar{v} \in C(\R^d)$ is the unique viscosity solution of the obstacle problem, 
\[
\bar{v} := \inf\{ \bar{v} \in C(\R^d) | \bar{v} \geq -G \mbox{ and } \bar{F}_\eta(D^2 (\bar{v}+G)) \leq 0\}.
\]
\end{theorem}
We would like to emphasize that essentially {\it any} well-posed PDE involving the operator $\bar{F}_{\eta}$ can be realized as the scaling limit of sandpiles with the arguments in this paper.

\subsection{Sandpiles with $\E(\eta(0)) \leq 2 d -1$} 
\begin{figure} 
\begin{centering}
  \includegraphics[width=0.5\textwidth]{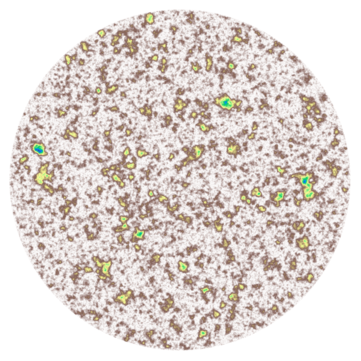}
 \caption{A heat map of the odometer function for a Bernoulli(0,4,0.528) initial sandpile started in a circle of radius $6 \cdot 10^3$ with the open boundary condition.}
  \label{fig_percolation}
\end{centering}
\end{figure}

The high density assumption, $\E(\eta(0)) > 2 d -1$, was used in two places in the paper. The first was to ensure 
that after stabilizing $\eta$ in a sufficiently large initial domain what is left is close to a recurrent configuration. 
The second was to show that solutions to $\bar{F}_{\eta}(D^2 \bar{v}) \leq 0$ also satisfy $\bar{F}_0(D^2 \bar{v}) \leq 0$.  

For the first usage, we can replace the assumption on $\E(\eta(0))$ by assuming that after stabilizing in all large enough nested volumes and removing an $o(n^2)$ portion of the boundary, what remains is recurrent.  For example, for each $p \in [0,1]$,
the following random sandpile on $\Z^2$ has a scaling limit by our argument as it is always recurrent,
\[
\eta(x) =  \begin{cases} 
2  \mbox{ with probability $p$}  \\
4 \mbox{ with probability $1 - p$}.
\end{cases}
\]

For the second usage, it suffices to use the weaker bound $\E(\eta(0)) \geq d$. And in fact,  if $\E(\eta(0)) < d$, the sandpile is almost surely {\it stabilizable}, \cite{fey2009stabilizability}. This implies, by conservation of density, (Lemma 2.10 in \cite{fey2009stabilizability}, Lemma 3.2 in \cite{levine2016divisible}), that the stable sandpiles converge weakly* to $\E(\eta(0))$ and so $\bar{v}_n \to 0$. 

This still leaves unaddressed sandpiles with $\E(\eta(0)) \in [d, 2d-1]$ which are not stabilizable, but also not close to a recurrent configuration. We believe, but cannot prove, that all such sandpiles have odometer functions with subquadratic growth.  See Figure \ref{fig_percolation} for an example of what could be such a sandpile.

\subsection{Characterizing the effective equation} 
Recently L. Levine, W. Pegden, and C. Smart characterized $\bar{F}_0$ on $\Z^2$ as the downwards closure of an Apollonian circle packing \cite{levine2016apollonian, levine2017apollonian}. Then, W. Pegden and C. Smart explained the microscale structure of the sandpile on $\Z^2$ by establishing a rate of the convergence to the continuum PDE and showing pattern stability \cite{pegden2017stability}. 

Analogous results for $\bar{F}_{\eta}$ are currently out of reach. Lemma \ref{regularity} shows that solutions to $\bar{F}_\eta(D^2 v ) \leq 0$ also satisfy $\bar{F}_0(D^2 v) \leq 0$. Numerical evidence also indicates that $\bar{F}_{\eta}$ is not always the Laplacian; one reason for this may be the extra log factor in the mixing time of the sandpile Markov chain, see the recent work of B. Hough, D. Jerison, and L. Levine \cite{hough2019sandpiles}. Their proof exploits so-called {\it toppling invariants} 
which describes sandpile quantities that are conserved under toppling (see also \cite{dhar1995algebraic}). We speculate that a characterization of the scaling limit will involve an interplay between these invariants and the choice of random distribution.

\section*{Acknowledgements}
I am grateful to Charles K. Smart for suggesting the program in \cite{armstrong2014quantitative}, patiently providing essential advice throughout this project, and carefully reviewing a previous draft of this paper.
I am also grateful to Steven P. Lalley for useful conversations, encouragement, and first introducing me to this problem. I thank Lionel Levine for generous, detailed comments on a previous draft 
and for helpful discussions. I also acknowledge Khalid Bou-Rabee, Nawaf Bou-Rabee, Gregory Lawler, and Micol Tresoldi for inspiring conversations.  The anonymous referee and editor also provided detailed feedback which led to a much improved exposition.

\bibliography{random_piles}{}
\bibliographystyle{amsalpha}

\end{document}